\documentclass[a4paper,11pt]{article}
\usepackage[UKenglish]{babel}
\usepackage[utf8]{inputenc}
\usepackage{graphicx}
\usepackage{amsmath}
\usepackage{amsthm}
\usepackage{amsfonts}
\usepackage{graphicx}
\usepackage{dsfont}
\usepackage{color}
\usepackage{hyperref}
\usepackage[top=2cm, bottom=2cm, left=2cm, right=2cm]{geometry}
\usepackage{stmaryrd}
\usepackage{version}
\usepackage{accents}
\usepackage{enumitem}
\usepackage{url}

\newcommand{\N}{\mathbb{N}}

\newcommand{\R}{\mathbb{R}}

\newcommand{\Z}{\mathbb{Z}}
\newcommand{\F}{\mathcal{F}}
\newcommand{\T}{\mathcal{T}}
\newcommand{\Pro}{\mathbb{P}}
\newcommand{\Esp}{\mathbb{E}}
\newcommand{\Var}{\mathrm{Var}}
\newcommand{\Cov}{\mathrm{Cov}}
\renewcommand{\det}{\mathrm{det}}

\newcommand{\eps}{\varepsilon}

\newcommand{\Leb}{\mathrm{Leb}}
\renewcommand{\d}[1]{\ensuremath{\operatorname{d}\!{#1}}}
\newcommand{\bigo}{\mathrm{O}}
\newcommand{\ubar}[1]{\underaccent{\bar}{#1}}
\def\quotient#1#2{%
	\raise0.5ex\hbox{$#1$}\Big/\lower0.5ex\hbox{$#2$}%
}

\theoremstyle{plain}
\newtheorem{prop}{Proposition}[section]
\newtheorem{theorem}[prop]{Theorem}
\newtheorem{coro}[prop]{Corollary}
\newtheorem{lem}[prop]{Lemma}

\theoremstyle{definition}
\newtheorem{defi}[prop]{Definition}

\theoremstyle{remark}
\newtheorem{rem}[prop]{Remark}

\title{\LARGE \bf Hydrodynamic Limit of a $(2+1)$-Dimensional Crystal Growth
Model in the Anisotropic KPZ Class.}
\author{Vincent Lerouvillois }
\date{}

\newcommand{\Addresses}{{
		\bigskip
		\footnotesize

		\textsc{Univ Lyon, Universit\'e Claude Bernard Lyon
			1, CNRS UMR 5208, Institut Camille Jordan, F-69622 Villeurbanne,
			France}\par\nopagebreak
		\textit{E-mail address}:
		\texttt{lerouvillois@math.univ-lyon1.fr}
}
}

\begin{document}
\maketitle

\begin{abstract}
	We study a model, introduced initially by Gates and Westcott
		\cite{gates1995stationary} to describe crystal growth evolution, which
		belongs to the Anisotropic KPZ
	universality class~\cite{prahofer1997exactly}. It can be thought of as a
	$(2+1)$-dimensional generalisation of the well known (1+1)-dimensional
	Polynuclear Growth Model (PNG). We show the full hydrodynamic limit
	of this process i.e the convergence of the random interface height
	profile after ballistic space-time scaling to the viscosity solution
	of a Hamilton-Jacobi PDE: $\partial_tu = v(\nabla u)$ with $v$ an
	explicit non-convex speed function. The convergence holds in the
	strong almost sure sense.
\end{abstract}

\section{Introduction}

Crystal growth belongs to a wider class of random
interface growth phenomena that appear naturally in
physics and biology \cite{barabasi1995fractal}. Trying to better
understand the behavior of these natural phenomena is a source of
interest in itself. On the other hand,  random growth models mainly
caught the attention of mathematicians in the last couple of decades
because of their conjectural universality properties and relation
with the KPZ
(Kardar-Parisi-Zhang) equation \cite{kardar_dynamic_1986} which
presumably encodes their long-time fluctuation behavior (see e.g
\cite{corwin2012kardar,ferrari2010random,quastel_introduction_2011}
for reviews on the topic in dimension $(1+1)$ and \cite{toninelli20172+} in
dimension $(2+1)$).

To fix ideas, the microscopic $d$-dimensional interface is typically
modelled by the graph of a discrete height function
$h: \Z^d \times \R_+ \to \Z$ (here, $\mathbb R_+$ represents the  time
variable) and evolves according to an asymmetric Markovian dynamic which is
often related to interacting particles systems. The transition rates are
assumed to
depend only on height gradients, so that the dynamics is invariant  by vertical
translations of the interface.  The first problem one
may address is the law of large numbers or \emph{hydrodynamic limit},
i.e the typical macroscopic behavior of the randomly evolving height function.
Under
space-time ballistic rescaling of the form
$n^{-1}h(\lfloor nx \rfloor,nt)$, the height profile  is expected to
converge to the solution of a first-order non-linear PDE of
Hamilton-Jacobi type:
\begin{equation}\label{eq:HJ-intro}
\partial_tu = v(\nabla u),
\end{equation}
where the growth velocity $v$ only depends on the slope and not
on $u$ itself since the model is vertically translation
invariant. Next, and more challengingly, comes the study of
\emph{fluctuations}, i.e the behavior of the discrete height function
around its hydrodynamic limit. The large-scale
fluctuations are expected to look qualitatively like the
solution of the KPZ equation and in particular share the same universal
characteristics exponents. Most results in this direction are established for
$d=1$. In dimension two, growth models are conjectured
\cite{wolf1991kinetic} to fall into two universality classes
depending on the convexity properties of $v$. When $v$ is strictly
convex (or
concave), we speak simply of KPZ universality class:
it is predicted and numerically observed that fluctuations grow
like $t^{\beta}$ with a universal exponent $\beta>0$ and
spatial fluctuations at equilibrium grow with a ``roughness
exponent'' $\alpha = 2\beta/(\beta+1)$. When the Hessian of $v$ has
signature $(+,-)$ the model is conjectured to belong to the so-called
\emph{Anisotropic} KPZ (AKPZ) class where spatial and temporal fluctuations
are expected to grow logarithmically and  spatial fluctuations to scale
to a Gaussian Free Field, as is the case for the stochastic
heat equation with additive noise. One says that the non-linearity in
the KPZ equation is
\emph{irrelevant} in the AKPZ regime and
\emph{relevant} in the KPZ one.

The model we are considering in this paper was introduced by
Gates and Westcott in \cite{gates1995stationary} to describe crystal
growth evolution and its stationary states. The interface can be
described by a height function $h :\left(\R \times \Z\right) \times \R_+ \to
\Z$,
semi-discrete in space and continuous in time, whose level lines are
piece-wise constant functions with $\pm1$ jumps. Even if we adopt a
different viewpoint, the Gates-Westcott dynamic can be viewed as a
multi-line generalisation of the PNG dynamic where each level line
follows simultaneously the PNG dynamic with ``kink/antikink creations''
suppressed
whenever two lines intersect. Although the PNG is a solvable model
that can be mapped to the problem of the longest increasing subsequence of a
random permutation, to
random polymers and to random matrix
ensembles (see \cite{ferrari2005one} for a nice review on the topic),
the Gates-Westcott dynamics induces non-trivial  interaction among level lines,
which makes the model harder to analyse. In
\cite{prahofer1997exactly}, Prähofer and Spohn identified a
slope-dependent family of stationary distributions for the
dynamic restricted to a bi-dimensional torus (note that Gates
and Westcott already computed equilibrium measures in
\cite{gates1995stationary} but only for a one-dimensional subset
of slopes $\rho$). In a certain thermodynamic limit of
large torus, they were able
to compute the slope-dependent growth velocity $v(\rho)$ at
stationarity. This
is the natural candidate for the speed function $v(\rho)$ in
\eqref{eq:HJ-intro}. As expected, the Hessian of $v$ has signature
$(+,-)$ everywhere so the model belongs to the AKPZ
universality class. The authors of \cite{prahofer1997exactly}
also showed that the spatial fluctuations at equilibrium are of
logarithmic order with respect to the distance between
points; this is typical of the two-dimensional
Gaussian Free Field. However, they didn't treat the temporal
fluctuations (also expected to grow logarithmically). Our contribution
to the study of the model is the rigorous proof of the hydrodynamic
limit starting from arbitrary initial condition. As an intermediate step, we
also get a logarithmic upper bound on fluctuation growth w.r.t. time in the
stationary states (see Lemma~\ref{lem:varh0}).

In the literature, most results about hydrodynamic limits in multi-dimensional spaces are given
for convex velocities $v$, where the viscosity solution of
\eqref{eq:HJ-intro} can be expressed in terms of the
variational Hopf-Lax formula. The strategy is to show that the
discrete height function enjoys a variational formula (sometimes
called "envelope property") at the microscopic level, which
passes to the limit thanks to the sub-additive ergodic
theorem. This applies e.g to the Corner Growth Model
\cite[Section 9]{seppalainen2008translation}, Ballistic deposition
\cite{seppalainen2000strong} and a wider family of grows models on
$\Z^d$ \cite{rezakhanlou2002continuum}, and yields existence of
such a hydrodynamic limit without providing an explicit
expression of the speed function $v$. The function $v$
can be explicitly identified when equilibrium measures are known, as
is the case for various one-dimensional models, such as ASEP and
PNG. For two-dimensional models in the AKPZ class, such envelope property
and Hopf-Lax
	formula cannot hold, otherwise the speed function in the
hydrodynamic limit would automatically result to be convex.

In his seminal article \cite{rezakhanlou2001continuum}, Rezakhanlou
introduced a different approach to hydrodynamic limit for growth
processes based on a compactness argument and on a list of conditions
that allow to identify any limit point with the unique viscosity
solution of \eqref{eq:HJ-intro}. This method does not require
convexity of $v$, but the only examples for which a full hydrodynamic
limit was proved in \cite{rezakhanlou2001continuum} are
one-dimensional where the structure of ergodic translation invariant
stationary measures is better understood. For $d \geq 2$, only a
partial result was obtained, namely, that any limit in distribution of
the rescaled height profile is concentrated on a set of viscosity
solutions of Hamilton-Jacobi equations with a possibly random speed
function. However, a precise description of equilibrum measures is
available for some of these models (e.g the Gates-Westcott model
\cite{prahofer1997exactly} and models related to the two-dimensional
dimer model
\cite{ferrari2008anisotropic,borodin2014anisotropic,corwin2016stationary,toninellistat,chhita20192+}
where the stationary measures are given by translation invariant Gibbs
measures on perfect matchings \cite{kenyon2006dimers}). Inspired by
Rezakhanlou's
technique, Zhang obtained the first full hydrodynamic limit
\cite{zhang2018domino} for a $(2+1)$-dimensional growth
model. Specifically, he considered the dimer shuffling-algorithm,
whose stationary distributions are given by weighted random dimer
configurations on $\Z^2$. Let us also mention the works
\cite{borodin2014anisotropic,legras2017hydrodynamic} about a long jump
two-dimensional interlaced particle dynamic generalising the
Hammersley process. In \cite{borodin2014anisotropic}, the authors
showed the hydrodynamic limit starting from a very specific initial
condition (with a CLT for temporal fluctuations on scale $\log t$)
while in \cite{legras2017hydrodynamic}, the authors proved the
hydrodynamic limit either up to the first time when a shock appears,
or under the assumption of a convex initial profile
\cite{evans2014envelopes}.

The present article follows the main idea of
\cite{rezakhanlou2001continuum,zhang2018domino} in terms of proof structure.
The idea consists in constructing a sequence (labeled by the parameter $n$
associated to the ballistic rescaling)
of discrete random semi-groups associated to the rescaled microscopic
dynamic, showing compactness in some sense and identifying
the limiting continuous semi-group with the one associated with the unique
viscosity solution of the PDE. The identification relies both on the
sufficient conditions given in
\cite{rezakhanlou2001continuum} (summarised in Proposition
\ref{prop:condaxiom})
and on a precise analysis of the stationary processes. With respect to
\cite{rezakhanlou2001continuum,zhang2018domino}, non-trivial additional
difficulties we had to overcome in the proof of compactness are related to the
semi-continuous character of the model and to  unboundedness of the  slopes and
of the speed function. In particular, we had to control the evolution of
spatial
gradients (Proposition \ref{prop:difhGW}) while this was trivial in
\cite{rezakhanlou2001continuum,zhang2018domino} since gradients are bounded. To
do so, we related the height function along the first coordinate to the PNG
with
a random subset of Poissonian creations and used a representation in terms of
random directed polymers. Also, instead of showing tightness of probability
measures like in \cite{rezakhanlou2001continuum,zhang2018domino}, we showed
that,
for a certain topology, the sequence of random semi-groups is almost surely
contained in a (random) compact set and then proved almost sure uniqueness of
the possible sub-sequential limits. Let us emphasize that the
hydrodynamic limit we obtained is in the strong sense of almost sure
convergence (on an underlying probability space determined by the
Poissonian clocks).

The article is structured as follows. The Gates-Westcott model is
introduced in Section 2: we define the state space of admissible
height functions and its dynamic via a Poisson Point Process on
$\R \times \Z \times \R_+$ representing space-time locations of kink-antikink
creations. In section
3, we start by stating the main result: the hydrodynamic limit for the
height function (Theorem \ref{theo:principalGW}). Then, we remind
elements of Hamilton-Jacobi PDE theory and useful results on
equilibrium measures. The rest of the article is dedicated to the
proof of the main theorem (the strategy of the proof is briefly explained at
the end of Section 3). In Section 4, we first show elementary facts
about the microscopic dynamic and a fundamental property of locality
(Corollary \ref{coro:asymlocal}) and then construct the sequence
of random discrete semi-group mentioned above. Section 5 is about proving
compactness. A key step in this
proof is the control of random spatio-temporal gradients (Propositions
\ref{prop:controlsgrowth} and
\ref{prop:difhGW}). Then, we apply a Arzelà-Ascoli type theorem (Proposition
\ref{prop:ascoli}) and
show compactness of the sequence of discrete semi-groups. Finally, in
Section 6, we identify the limit points as the semi-group associated with the
unique viscosity solution of
\eqref{eq:hamilton-jacobiGW} thanks to
Proposition \ref{prop:condaxiom} and the results about equilibrium
measures.

\section{The Gates-Westcott model}

\subsection{Height function}

In this model, the surface will be described by a discrete height function $\varphi : \R \times \Z \to \Z$ which lives in the state space given as follows:
\begin{defi}\label{defi:statespace}
Let $\Gamma$ be the set of functions $h : \R \times \Z \to \Z$ satisfying the following two conditions:
\begin{enumerate}
\item For any $y \in \Z$, $x \mapsto h(x,y)$ is piece-wise constant with
a locally finite number of  $\pm 1$-valued jumps. By convention, we impose that the values
at discontinuity points make the function upper semi-continuous.
\item For any $x \in \R$, $h(x,y+1) - h(x,y) \in \{-1, 0\}$.
\end{enumerate}
\end{defi}

\noindent Because of condition 1, the discontinuities along direction $x$ can be of three different types:
\begin{itemize}
	\item \emph{kink}: $h(x,y) = h(x^-,y) = h(x^+,y) + 1$
\item \emph{antikink}: $h(x,y) = h(x^-,y) + 1 = h(x^+,y)$
\item \emph{kink-antikink pair}: $h(x,y) = h(x^-,y) + 1 = h(x^+,y) + 1$.
\end{itemize}

A height function looks like a stack of terraces seen from a plane
(see Figure \ref{fig:GW}), the edges of each terrace along the $x$
direction corresponding to the kinks and antikinks of the height
function.  Due to the first condition in Definition
\ref{defi:statespace}, each function $h(\cdot,y)$ is entirely
determined by the position of its kinks and antikinks and its height
at any point $x_0 \in \R$. In other words, the kinks and antikinks
define the variations of the height function along the $x$ direction.

\begin{rem}
  In the article \cite{gates1995stationary} of Gates and Westcott,
  condition 2 was replaced by the height function being integer-valued and
  non-decreasing
  along the $y$ direction so that arbitrary slopes could be allowed
  (which is physically more realistic). However, there exists a
  one-to-one correspondence between height functions according to
  these two definition variants, as explained in
  \cite[p.91]{prahoferthesis}.
\end{rem}

\subsection{Dynamic}

Let $\omega$ be a Poisson Point Process of intensity $2$ on
$\R \times \Z \times \R_+$ seen as a random, locally finite, set of points
in $\R \times \Z \times \R_+$ that will be called
\emph{creations}. Starting from a configuration in the state space
$\Gamma$, the Gates-Westcott dynamic is defined by three rules: the
first two are deterministic while the last one is random.
\begin{itemize}
\item \emph{Lateral Expansion} : each terrace exands laterally at
  speed $1$, i.e. each kink (resp. antikink) of the height function is
  moved at speed $+1$ (resp $-1$) along the $x$ direction.
\item \emph{Annihilation} : whenever two terraces meet, they merge,
  i.e. whenever a kink and an antikink meet, they annihilate each
  other.
\item \emph{Creation} : If $(x,y,t) \in \omega$, then the height $h$
  at $(x,y)$ increases by one at time $t$ if the height function
  obtained remains in $\Gamma$.  In other words, a kink-antikink pair
  is created at time $t$ and at space position $(x,y)$ if the height
  function remains in $\Gamma$ after the transition, i.e. if
  $h(x,y-1,t^-) - h(x,y,t^-) = 1$, $h(x,y,t^-) - h(x,y+1,t^-) = 0$ and
  if there is no preexisting discontinuity of  $h(\cdot,y,t^-)$ at $x$.  Note
  that the last condition is verified
  with probability $1$,  since the discontinuities
  are locally finite hence countable for any function in~$\Gamma$.
\end{itemize}

\begin{figure}[htbp]
\begin{center}
\includegraphics[scale=0.8]{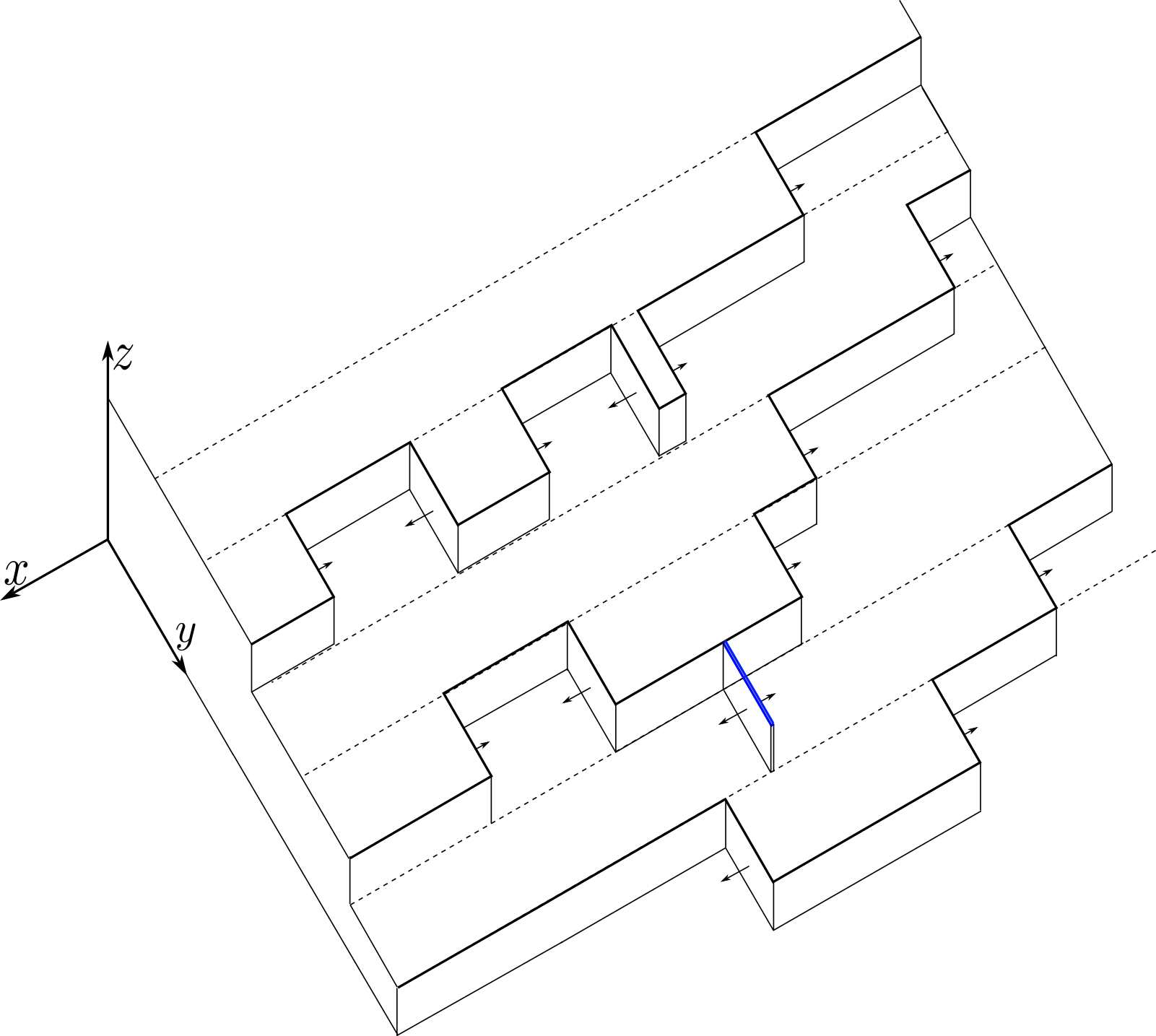}
\end{center}
\vspace{-0.5cm}
\caption{ \footnotesize A section of a height function. The lateral expansion is indicated by arrows. A newly created terrace expansion is shown in blue. }
\label{fig:GW}
\end{figure}

\begin{rem}\label{rem:welldefined}
		As usual in interacting particle systems, some care has to
		be taken to ensure that the process is well defined on the
		infinite lattice. If we worked in a finite domain, there would be
		a finite number of creations in finite time intervals and we could
		know the height function deterministically up to the first time of
		creation, determine whether this creation occurs or not and repeat
		the procedure inductively on the number of creations. On the
		infinite lattice it makes no sense to look at ``the first
		creation'' but existence and locality of the dynamics can be
		proven by a modification of the classical disagreement percolation argument
		\cite[Sec. 3.2]{martinelli1999lectures} used for Glauber dynamics on
		infinite graphs. Namely, suppose we want to determine the
		evolution of $h(x,y,t)$ for all
		$(x,y,t)\in [a,b]\times \llbracket c, d \rrbracket \times [0,T]$. Since kinks/antikinks move with
		speed $1$, we see that creations whose $x$-coordinate is outside
		$[a-T,b+T]$ do
		not matter. Also, let  $y^-$ (resp. $y^+$) be the largest integer smaller
		than $c$ (resp. the smallest integer larger than $d$)
		such that there are no creations at
		$(x,y,t)\in [a-T,b+T]\times \{y^\pm\} \times [0,T]$, then the creations
		that happen for $y>y^+$ or $y<y^-$ also do not matter and then
		the evolution of $h(x,y,t)$ for all $(x,y,t)\in [a,b]\times \llbracket c, d \rrbracket \times [0,T]$ is
		determined by the finitely many creations in the bounded domain
		$(x,y,t)\in [a-T,b+T]\times \llbracket y^-,y^+\rrbracket\times
		[0,T]$.Finally, for any $a<b \in \R$ and $T>0$, the random variables $y^\pm$ are almost surely finite.
		Later (cf. Proposition \ref{prop:prolinGW}) we will prove a more quantitative
		locality statement: the height at a
		point up to time $T$ is determined by creations that occur in a domain that, with
		high probability, grows linearly with $T$.
\end{rem}

\begin{rem}\label{rem:lines} The Gates-Westcott model can be equivalently
described in
  terms of level lines of the height function (i.e. the bold lines
  drawn by the terraces edges seen from above in figure
  \ref{fig:GW}) as explained  in
  \cite{gates1995stationary,prahofer1997exactly,prahoferthesis}. From
  this point of view, the dynamic is nothing but the Polynuclear
  Growth (PNG) Model dynamic \cite{ferrari2005one} applied simultaneously to each level line, creations being  suppressed whenever two lines intersect.
\end{rem}

\section{The main result}

\subsection{Hydrodynamic limit}

First of all, let us introduce a few definitions and notations.  We
denote by $\Omega$ the set of locally finite subsets of
$\R \times \Z \times \R_+$ endowed with the $\sigma$-algebra and the
probability measure induced by a Poisson Point Process of intensity
$2$ on $\R \times \Z \times \R_+$. For all $\omega \in \Omega$, for all
admissible height function $\varphi \in \Gamma$ and for all
$(x,y,t) \in \R \times \Z \times \R_+$, we define
\begin{equation}
 h(x,y,t;\varphi,\omega)
\end{equation}
as the height function at time $t$ obtained by applying the
Gates-Westcott dynamic described in the previous section with initial
height profile $\varphi$ and creations $\omega$. Let us also define
the continuous state-space
\begin{equation}\label{eq:gammabar}
\bar{\Gamma} := \left\{f \in \mathcal{C}\left(\R^2\right), \ \forall x \in \R,
\ \forall y_1 \leq y_2 \in \R, \ f(x,y_2) - f(x,y_1) \in [-(y_2-y_1),0] \right\}.
\end{equation}
Notice that a continuously differentiable function on $\R^2$ is in
$\bar{\Gamma}$ if and only if its gradient takes values in
$\R \times [-1,0]$.
\begin{theorem}\label{theo:principalGW}
  Let $(\varphi_n)_{n \in \N} \in \Gamma^{\N}$ be a sequence of
  admissible initial height functions approaching a continuous
  function $f \in \bar{\Gamma}$ in the following sense:
\begin{equation}\label{eq:hypcondinitialGW}
\forall R>0 \qquad \sup_{|x|,|y| \leq R} \left| \frac{1}{n} \varphi_n(nx,\lfloor ny \rfloor) - f(x,y) \right| \underset{n \to \infty}{\longrightarrow} 0 \ .
\end{equation}

\noindent Then, for almost all $\omega$ in $\Omega$,

\begin{equation}\label{eq:limitehydroGW}
\forall T>0 \quad \forall R>0 \qquad  \sup_{|x|,|y| \leq R, t \in [0,T]}
\left| \frac{1}{n} h(nx,\lfloor ny \rfloor , nt;\varphi_n,\omega) - u(x,y,t)
\right| \underset{n \to \infty}{\longrightarrow} 0 \ ,
\end{equation}
where $u$ is the unique viscosity solution of the Hamilton-Jacobi equation
\begin{equation}\label{eq:hamilton-jacobiGW}
\left\{
\begin{aligned}
\partial_t u  & =  v( \nabla u) \\
u(\cdot,\cdot,0) & = f \ ,
\end{aligned}
\right.
\end{equation}
with
\begin{equation}\label{eq:speedfunction}
v(\rho_1,\rho_2)= \frac{1}{\pi} \sqrt{ \pi^2 \rho_1^2 + 4 \sin^2(\pi \rho_2)} \ .
\end{equation}
\end{theorem}

\begin{rem}
	For any continuous function $f \in \bar{\Gamma}$, we can always find a
	sequence of functions $\varphi_n \in \Gamma$ approaching $f$ in the sense
	of
	\eqref{eq:hypcondinitialGW} as we will show later in Proposition
	\ref{prop:defphiGW}.
\end{rem}

Before proving this theorem, we will remind the definition of the
viscosity solution of Hamilton-Jacobi equations in the next section,
explain why it is unique and state sufficient conditions to identify
it. In Section \ref{sec:equilibrium}, we will present useful results
about equilibrium measures taken from
\cite{prahofer1997exactly,prahoferthesis}, where the speed function in
\eqref{eq:speedfunction}  is also computed.

\subsection{Viscosity solutions of Hamilton-Jacobi equations}

In this section, we briefly recall some elements of the theory of
Hamilton-Jacobi Partial Differential Equations. In order to show
Theorem \ref{theo:principalGW}, all we need to know about viscosity
solutions is gathered in Theorem \ref{theo:unique} and in Proposition
\ref{prop:condaxiom}. The interested reader can find more background
and motivations about Hamilton-Jacobi equations in the monography
\cite{evans2010partial} for instance.

Given $f, v \in \mathcal{C}(\R^d)$, we consider the following first order PDE:
\begin{equation}\label{eq:H-J}
\left\{
\begin{array}{rcll}
\partial_t u & = & v(\nabla u) & \text{on } \R^d \times (0, +\infty)\\
u(\cdot,0) & = & f & \text{on } \R^d.
\end{array}
\right.
\end{equation}
\noindent Under some further regularity conditions on $v$ and $f$, it
is possible to apply the method of characteristics to obtain a local
classical solution. In general, whatever the regularity of $v$ and
$f$, shocks for $\nabla u$ appear in finite time and the solution is
no more differentiable. In order to give a definition of solution that
is global in time, we introduce the classical concept of \emph{viscosity
solution} that
guarantees existence and uniqueness under suitable assumptions.

\begin{defi}\label{defi:solvisc}
	We say that $u : \R^d \times [0,T] \to \R$ is a \emph{viscosity solution}
	of
	\eqref{eq:H-J} on $\R^d\times [0,T]$ if~$u$ is continuous, $u(.,0) = f$ and
	$u$ is
	both a subsolution and a supersolution.

	A function $u$ is a \emph{subsolution} (respectively a
	\emph{supersolution}) if for all~$\phi \in
	\mathcal{C}^{\infty}(\R^d \times(0,T))$ and all
		$(x_0,t_0) \in \R^d \times (0,T)$ such that $\phi(x_0,t_0) =
		u(x_0,t_0)$
		and $\phi \geq u$ (resp. $\phi \leq u$) on a neighbourhood of
		$(x_0,t_0)$, the following inequality holds:
		\begin{equation}
		\begin{aligned}
		\partial_t \phi(x_0,t_0) &\leq v(\nabla \phi(x_0,t_0))\\
			(\text{resp. } \partial_t \phi(x_0,t_0) &\geq v(\nabla
			\phi(x_0,t_0)) \ ).
		\end{aligned}
		\end{equation}

\end{defi}

We won't address the question of general existence of viscosity
solutions because, in our case, we will show existence by proving that the hydrodynamic limit is indeed a solution of \eqref{eq:hamilton-jacobiGW}. However, a result of uniqueness will be needed to identify the potential limit points. The following Theorem shown by Ishii can be obtained as a corollary
of  \cite[Th. 2.5]{ishii1984uniqueness}.
\begin{theorem}\label{theo:unique}
  If $v$ is globally Lipschitz, there is at most one viscosity
  solution of~\eqref{eq:H-J} on~$\R^d\times[0,T]$.
\end{theorem}
Since the function
$v$ in \eqref{eq:speedfunction} is globally Lipschitz, there is at most one
viscosity
solution of
\eqref{eq:hamilton-jacobiGW}.

The next proposition gives sufficient conditions to identify the viscosity
solution of \eqref{eq:hamilton-jacobiGW}. Even if it is stated for the special
case of functions
 living in the two-dimensional continuous state-space $\bar{\Gamma}$ defined in
 \eqref{eq:gammabar} and for the speed function $v$ defined in
 \eqref{eq:speedfunction}, it can be easily extended to a more general
 framework.
\begin{prop}\label{prop:condaxiom}
	Let $T$ be a positive real number and $S(s,t,\cdot)_{0\leq s \leq t \leq T}$ be a family of functions from
	$\bar{\Gamma}$ into itself satisfying the following properties:
	\begin{enumerate}
		\item \emph{Translation invariance} : for all $f\in \bar{\Gamma}$, all $c \in
				\R$ and all~$s \leq t$,
		$$S(s,t,f+c) =	S(s,t,f) + c.$$
		\item \emph{Monotonicity}: for all $s \leq t$, and all $f,g \in
				\bar{\Gamma}$,
		$$f \leq g \Rightarrow S(s,t,f)
		\leq
		S(s,t,g).$$
		\item \emph{Locality}: There exists $\alpha > 1$ such that for all
		$f,g \in \bar{\Gamma}$, all $s \leq t$, all $x\in \R^2$ and all $R\geq0$
		\[
		\sup_{z \in \mathcal{B}(x,R)} | S(s,t,f)(z) - S(s,t,g)(z) |
		\leq \sup_{z \in \mathcal{B}(x,R+\alpha(t-s))}| f(z) - g(z)|,
		\]
		where $\mathcal{B}(x,r)$ is the ball of centre $x$ and radius $r$ for
		the supremum norm on $\R^2$.
		\item \emph{Semi-group} : for all $r \leq s \leq t$ and all $f
				\in \bar{\Gamma}$,
		$$S(r,t,f) = S(s,t,S(r,s,f))  \text{ and } S(t,t,f) = f.$$
		\item \emph{Compatibility with linear solutions} : for all linear
                  function $f_{\rho} : x \mapsto \rho\cdot x$ with $\rho \in \R
                  \times [-1,0]$
                  and all $s \leq t$, $$S(s,t,f_{\rho}) = f_{\rho} + v(\rho)(t-s).$$
	\end{enumerate}

\noindent For any $f \in \bar{\Gamma}$, if $(x,t) \mapsto
	S(0,t,f)(x)$ is continuous, then it is a viscosity solution of
	\eqref{eq:H-J}.
\end{prop}

The proof of this proposition is postponed to appendix \ref{sec:sufcond}.

\subsection{Equilibrium measures}\label{sec:equilibrium}

In this section, we briefly remind a few facts about equilibrium
measures, following Prähofer and Spohn
\cite{prahofer1997exactly,prahoferthesis}.  They identified a family
of random height functions, whose spatial height differences
have a law that is translation-invariant with a slope parameter
  $\rho$ in $\R \times (-1,0)$, and are stationary with respect to
time (Gates and Westcott already treated the case
$\rho \in \{0\}\times (-1,0)$ when they introduced their model in
\cite{gates1995stationary}).  Prähofer and Spohn also computed the
stationary growth
speed~$v(\rho)$ which gives the candidate speed function of the
Hamilton-Jacobi equation \eqref{eq:speedfunction} in Theorem
\ref{theo:principalGW} and showed that the variance of spatial height
differences behaves logarithmically. To do so, they used fermionic Fock
space tools to carry out a fine analysis of the equilibrium
measures. Let us sum up useful results, most of which can be recovered or easily deduced
from \cite[Section 6]{prahoferthesis} and others will be detailed in Appendix \ref{sec:kak_cor_eq}.

The starting point of Prähofer and Spohn
  \cite{prahofer1997exactly,prahoferthesis} is the analysis of the
  Gates-Westcott model in a periodized setting, i.e. on a torus
  $[-M,M) \times \llbracket -N, N-1 \rrbracket$. Let us remark that, even though we use different notations, we follow the construction of \cite{prahoferthesis} rather than \cite{prahofer1997exactly} in which a more complicated "twisted" periodic boundary condition is considered (both constructions lead to the same results in the infinite volume limit of the torus). The allowed height profiles have space gradients
  that are periodic with horizontal period $2N$ and vertical period
  $2M$.  They evolve according to the \emph{periodised Gates-Westcott
    dynamic} i.e the Gates-Wescott dynamic with periodised Poissonian
  creations $[\omega]^{M,N}$ defined from $\omega$ as follows:
\begin{equation}\label{eq:poissonperio}
(x,y,t) \in [\omega]^{M,N} \Leftrightarrow ([x]^M, [y]^N,t) \in \omega,
\end{equation}
where $[x]^M$ is the unique number in $[-M,M)$ equal to $x$ modulo
$2M$ and similarly for $[y]^N$. In \cite[Section
6.2]{prahoferthesis}, the author defined a family of random height
functions (see \cite[equation (6.9)]{prahoferthesis}) taking values in $\Gamma$, whose law is
indexed by weight parameters $\eta^\pm$ on antikinks and on
kinks and a slope parameter along $y$ (related to the density of
level lines). The space gradients of these functions are $2M,2N$
  periodic and their law is translation invariant and time
  stationary. Fixing properly the weights $\eta^\pm$ and the line
  density, one can guarantee that the average slope approaches any
  fixed $\rho$ in $\R \times (-1,0)$ when the size of the torus tends
  to infinity (sending first $M$  and then $N$ to infinity).  We call
  then $\varphi_{M,N,\rho}$ the stationary periodized profile with
  limit slope $\rho$, so that
\begin{equation}\label{eq:moyequi}
\forall (x,y) \in \R \times \Z \qquad \lim_{N \to \infty} \lim_{M \to \infty} \Esp \left[ \varphi_{M,N,\rho}(x,y)
\right] = \rho \cdot (x,y)
\end{equation}%
(we are fixing here $ \varphi_{M,N,\rho}(0,0)=0$) and stationarity translates into
\begin{equation}\label{eq:stationary}
	\forall t \geq 0 \qquad
	h(\cdot,\cdot,t;\varphi_{M,N,\rho},[\omega]^{M,N})
	-h(0,0,t;\varphi_{M,N,\rho},[\omega]^{M,N})
	 \overset{\mathrm{law}}{=}
	\varphi_{M,N,\rho}(\cdot,\cdot).
	\end{equation}

      In \cite{prahofer1997exactly,prahoferthesis}, the authors showed
      that the joint probability density of kinks, antikinks and
      occupation variables (i.e the set of $(x,y)$ such that
      $\varphi_{M,N,\rho}(x,y+1)-\varphi_{M,N,\rho}(x,y)=-1$) has a
      determinantal structure and identified the associated
      kernel. When the size of the torus tends to infinity, the
      expression of this kernel somehow simplifies (see \cite[equation
      (6.20)]{prahoferthesis}).  Also, the average
      growth velocity is equal to the sum of the kink and antikink
      densities (independent of time by stationarity),  and one obtains
      \cite[Equation (6.24)]{prahoferthesis}:
\begin{equation}\label{eq:speedequilibrum}
\forall (x,y,t) \in \R \times \Z \times \R_+ \qquad \lim_{N \to \infty} \lim_{M \to \infty} \Esp  \left[
h(x,y,t;\varphi_{M,N,\rho},,[\omega]^{M,N})  \right] = \rho \cdot (x,y) +
v(\rho)\,  t,
\end{equation}
with  $v(\rho)$ as in \eqref{eq:speedfunction}.

Prähofer and Spohn also computed the covariance (or "structure
  function") between kinks, antikinks and occupation variables (see  \cite[Equation (6.30)]{prahoferthesis} and \cite[Equation (27) and
  (29)]{prahofer1997exactly}). They deduced that, after taking the infinite volume limit,
the variance of the height difference at equilibrium is equivalent to
$\pi^{-2}\log(\|(x,y)\|)$ as $\|(x,y)\|\to\infty$, but under the technical constraint
that $y/x$ is
constant or $x=\mathrm{o}(y)$. For our purposes, we will simply need the following upper bound that holds without technical restriction on
$x,y$:
  \begin{equation}\label{eq:logfluct}
  \lim_{N \to \infty} \lim_{M \to \infty} \Var \left(
  \varphi_{M,N,\rho}(x,y) \right) =
  \underset{\|(x,y)\| \to
  	\infty}{\bigo} \left(
  \log \left( \|(x,y)\| \right)\right).
  \end{equation}
	Equation \eqref{eq:logfluct} can be is easily shown by bounding the variance of $\varphi_{M,N,\rho}(x,y)$ by twice the sum of the variance of $\varphi_{M,N,\rho}(x,0)$ and the variance of $\varphi_{M,N,\rho}(x,y) - \varphi_{M,N,\rho}(x,0)$ (by Cauchy-Schwarz inequality) which grow logarithmically w.r.t $|x|$ and $|y|$, according to the asymptotic computations of Prähofer and Spohn.

	Finally, it can be shown that the kink/antikink covariance	decays like the inverse of the distance squared multiplied by a bounded
	oscillating term (an upper bound will be proven in Appendix \ref{sec:kak_cor_eq}). Note that this is similar to the large-distance behavior of
	dimer-dimer correlations in dimer models \cite{kenyon2009lectures}.
From this, it is easy deduce (see Appendix \ref{sec:kak_cor_eq}) that
\begin{equation}\label{eq:varkinks}
\lim_{N \to \infty} \lim_{M \to \infty} \Var \left(
N^{\pm}_{M,N,\rho}(\Lambda_R)
\right)  = \underset{R \to \infty}{\bigo} \left(
R^2 \log R\right),
\end{equation}
where
$N^{\pm}_{M,N,\rho}(\Lambda_R)$ is the number of
antikinks/kinks of $\varphi_{M,N,\rho}$ in the domain $\Lambda_R := [-R,R]
\times \llbracket -R , R \rrbracket$.


\section*{Strategy of proof of Theorem \ref{theo:principalGW}}
\addcontentsline{toc}{section}{Strategy of proof} The crucial point is
Proposition \ref{prop:condaxiom} which gives sufficient
conditions for identifying the viscosity solution of
\eqref{eq:hamilton-jacobiGW}. Most of
  these conditions are naturally satisfied by the microscopic
  Gates-Westcott dynamics, apart from the ``compatibility with linear
  solutions'' which requires a study of the process started from the translation invariant stationary measures, beyond what was obtained in
  \cite{prahofer1997exactly,prahoferthesis}. The rest of the proof is
  based on compactness arguments, that allow to show
  sub-sequential existence of $S(s,t,\cdot)$ as the limit of the random
  microscopic semi-group $S_n(s,t,\cdot,\omega)$,
  associated to the rescaled Gates-Westcott
  dynamics. At the end, one identifies the limiting continuous semi-group thanks to
  Proposition \ref{prop:condaxiom}. The main steps of the proof are
  summed up as follows:
\begin{enumerate}
	\item \emph{Construction of a sequence of random discrete semi-groups}
	$(S_n(s,t;.,\omega))_{0 \leq s \leq t \leq T, \, n \in \N}$ (that will be defined more precisely in Section \ref{sec:defopGW}, Definition \ref{defi:S_n}):
	\[
	S_n(s,t;.,\omega) :
	\left\{
	\begin{aligned}
	\bar{\Gamma}  & \longrightarrow  \mathcal{F}(\R^2) \\
	f & \longmapsto \frac{1}{n}h(n.,\lfloor n. \rfloor,
	n(t-s);\varphi_n^f,\tau_{ns}\omega),
	\end{aligned}
	\right.
	\]
	with $\varphi_n^f \in \Gamma$ approaching $f$ in the sense of
        \eqref{eq:hypcondinitialGW}:
        $\|\frac{1}{n}\varphi_n^f(n.,\lfloor n. \rfloor) - f
        \|_{\infty} \leq 2/n$, $\tau_{ns}\omega$ is the time translation by $ns$ of $\omega$ defined later in \eqref{eq:translation} and where $\F(\R^2)$ is the set of
        functions from $\R^2$ to $\R$.  The function
        $S_n(s,t,f;\omega)$ should be thought of as the
        rescaled height function following the dynamic starting close
        from the continuous initial profile $f$ and with Poissonian
        creations taken between the macroscopic times $s$ and $t$.

	\item \emph{Compactness} : Show that there exists a subset $\Omega_0
	\subseteq \Omega$ of probability $1$ such that for any fixed $\omega \in
	\Omega_0$, from any subsequence $(n_k)_{k \in \N}$, we can extract a
	subsubsequence $(n_{k_l})_{l \in \N}$ such that for any function
	$f\in \bar{\Gamma}$, $(S_{n_{k_l}}(\cdot,\cdot;f,\omega))_{l \in \N}$ (seen
	as a sequence of
	functions from $\{(s,t)\in[0,T]^2, \, s \leq t\}$ to $\F(\R^2)$) converges
	for the topology of uniform convergence on all compact sets to a certain
	limiting function $S(\cdot,\cdot;f,\omega)$ which is continuous in space
	and time. The proof
	relies on a control of spatio-temporal height differences and on an
	adaptation of Arzelà-Ascoli's Theorem (see
	Proposition \ref{prop:ascoli}).

	\item \emph{Identification of the limit} : Show that any such limit
	$S(.,.;.,\omega)$
	satisfies the sufficient conditions of Proposition \ref{prop:condaxiom}
	thus $(x,y,t) \mapsto S(0,t;f,\omega)(x,y)$ is the unique
	viscosity solution of \eqref{eq:hamilton-jacobiGW}. The knowledge on
	equilibrium measures explained in section \ref{sec:equilibrium} will be
	used to show compatibility with linear solutions.

      \end{enumerate}

\section{Construction of a sequence of random discrete semi-groups}

Let us start by defining, for later use, the set of creations that  lead to an
actual height increase.
\begin{defi}\label{defi:omegaphi}
	For all $\omega \in \Omega$ and all $\varphi \in \Gamma$,
	we define the subset of $\emph{effective creations}$:
	\begin{equation}
	\omega^{\varphi} := \{ (x,y,t) \in \omega: \ h(x,y,t;\varphi,\omega) -
	h(x,y,t^-;\varphi,\omega) = 1 \}.
	\end{equation}
	It is a
		 subset of $\omega$ that depends (non trivially) only on $\varphi$
		and $\omega$.  For all $y\in \Z$, we define  the restriction of
		$\omega^{\varphi}$ and $\omega$ to line $y$:
	\begin{eqnarray}
	\omega^{\varphi}_y & := &\omega^{\varphi} \cap \left(\R \times \{y\} \times
	\R_+ \right)\\
	 \omega_y & := & \omega \cap \left(\R \times \{y\} \times
	 \R_+ \right)
	\end{eqnarray}
	By abuse of notation,
	we will
	see $\omega_y$ and $\omega^{\varphi}_y$ as subsets of $\R^2$.
\end{defi}

\subsection{Useful properties of the microscopic dynamic}\label{sec:microscopic}

In this section, we present useful properties satisfied by the
microscopic dynamic that will be useful to apply Proposition
\ref{prop:condaxiom} later on but also to show compactness.

\begin{lem}[Translation invariance]\label{lem:transinv}
	For all constant $m \in \Z$, all $\omega \in \Omega$, all $\varphi \in
	\Gamma$ and all $t \in \R_+$,
	\begin{equation}
	h(\cdot,\cdot,t;\varphi+m,\omega) = h(\cdot,\cdot,t;\varphi,\omega)
	+ m.
	\end{equation}
\end{lem}
\begin{proof}
	Having fixed $\omega$, by definition, the Gates-Westcott dynamic only
	depends on the height differences of the initial height function
	(kinks/antikinks and
	relative height differences along $y$). Therefore, the temporal height
	growth
	$h(\cdot,\cdot,t;\varphi,\omega)-\varphi$
	depends on $\varphi$ only through its spatial height differences hence is
	invariant
	by addition of a constant~$m$ to the initial function $\varphi$.
\end{proof}

\begin{lem}[Monotonicity]\label{lem:monotoneGW}
	For all $\varphi_1,\varphi_2 \in \Gamma$, for all $\omega \in \Omega$, and
	all $t\in
	\R_+$,
	\begin{equation}
	\varphi_1 \leq \varphi_2  \Rightarrow  h(.,.,t;\varphi_1,\omega) \leq
	h(.,.,t;\varphi_2,\omega) .
	\end{equation}
\end{lem}
\begin{proof}
	As explained in Remark \ref{rem:welldefined}, the dynamic can be defined
	locally and thus it
	is enough to show this Lemma when there are finitely many creations. It is
	not hard to show that the
	deterministic part of the dynamic (lateral expansion and annihilation) is
	non-decreasing with respect to the initial condition. We just have to check
	that any creation preserves monotonicity.

	Suppose that there
	is a
	creation at $(x,y,t)$ and that $h(.,.,s;\varphi_1,\omega) \leq
	h(.,.,s;\varphi_2,\omega)$ for $s < t$. Let us show that
	$h(x,y,t;\varphi_1,\omega) \leq
	h(x,y,t;\varphi_2,\omega)$. If $h(x,y,t^-;\varphi_1,\omega) <
	h(x,y,t^-;\varphi_2,\omega)$, then there is nothing to show since the
	height can
	only jump by one after a creation. If $h(x,y,t^-;\varphi_1,\omega) =
	h(x,y,t^-;\varphi_2,\omega)$ and if the creation is allowed for the dynamic
	starting from~$\varphi_1$, then so it is for the one starting from
	$\varphi_2$ because
	\begin{flalign*}
	&&h(x,y-1,t^-;\varphi_2,\omega) -
	h(x,y,t^-;\varphi_2,\omega) &\geq
	h(x,y-1,t^-;\varphi_1,\omega) - h(x,y,t^-;\varphi_1,\omega) = 1 &&\\
	\rlap{\text{and}} && h(x,y,t^-;\varphi_2,\omega) -
	h(x,y+1,t^-;\varphi_2,\omega) &\leq
	h(x,y,t^-;\varphi_1,\omega) - h(x,y+1,t^-;\varphi_1,\omega) = 0.&&
	\end{flalign*}
	In any case, the monotonicity is preserved
	after a creation.
\end{proof}

For $\omega \in \Omega$, we define $\tau_s \omega$, the time translation by   $s$ of $\omega$
as
follows:
\begin{equation}\label{eq:translation}
\forall (x,y,t) \in \R \times \Z \times \R_+, \; (x,y,t) \in \tau_s \omega
\Leftrightarrow (x,y,t+s) \in \omega.
\end{equation}

\begin{lem}[Markov property]\label{lem:markovGW}
	For all $\varphi \in \Gamma$, all $0 \leq s \leq t$ and all $\omega \in
	\Omega$,
	\begin{equation}
	h(.,.,t;\varphi,\omega) = h(.,.,t-s;h(.,.,s;\varphi,\omega),\tau_s \omega),
	\end{equation}
	and for all $0 \leq r \leq s \leq t$,
	\begin{equation}
	h(.,.,t-r;\varphi,\tau_r\omega) =
	h(.,.,t-s;h(.,.,s-r;\varphi,\tau_r\omega),\tau_s \omega).
	\end{equation}
\end{lem}
\begin{proof}
	From Remark \ref{rem:welldefined}, we can assume that $\omega$ contains
	finitely
	many points. In this case, the first point follows directly from the
	construction
	of the dynamic. The second point is obtained from the first point applied
	to $(s',t',\omega')=(s-r,t-r,\tau_r \omega)$.
\end{proof}

Next, as announced in Remark \ref{rem:welldefined}, we are going to show that
the dynamic on a bounded space-time domain only depends on
the initial height function and the creations on a bigger domain that grows
linearly with time with high probability. To make this statement precise, for
any $x \in \R^2$, $R\geq0$, $t \in \R_+$ and $\alpha >0$, let us define
\begin{equation}\label{eq:defA}
A_{x,R,t,\alpha} = \left\{ \omega \in \Omega, \quad
\begin{aligned}
& \forall \varphi,
\varphi'
\in \Gamma \quad \forall \omega' \in \Omega \\
& \text{if } \varphi=\varphi' \text{ on } \mathcal{B}(x,R+\alpha \, t)\text{
and
} \omega'=\omega \text{ on }
\mathcal{B}(x,R+\alpha \, t) \times [0,t] \\
& \text{then } \forall u \leq t, \ {h(\cdot,\cdot,\cdot;\varphi,
	\tau_{u}\omega)} = {h(\cdot,\cdot,\cdot;\varphi',
	\tau_{u}\omega')} \text{ on } \mathcal{B}(x,R) \times [0,t-u]
\end{aligned}
\right\}
\end{equation}
where the notation $\mathcal{B}$ abusively denotes the ball (for the supremum
norm) intersected
with
$\R \times \Z$.

\begin{prop}[Linear propagation of information]\label{prop:prolinGW}
	There exist constants $\alpha>1$ and $\gamma>0$, such that for all $R>0$,
	all $t \in \R_+$ and all $x \in \R^2$,
	\[
\Pro\left(A_{nx,nR,nt,\alpha} \right) \stackrel{n \to \infty} = 1 -
{\mathrm{O}}\left( e^{-\gamma t \, n} \right)\ .
\]
      \end{prop}

\begin{proof}
To lighten the notations, without loss of generality, we will assume that
$x=0$. The idea of the proof is the following. If the height functions differ on
$\mathcal{B}(0,R)\times[0,t]$ and if initial conditions and creations
agrees on $\mathcal{B}(0,R+\alpha t)$, then there must exists a "chain of
creations" of length at least $\alpha t$ (connecting $\mathcal{B}(0,R)$ to the
complement of $\mathcal{B}(0,R+\alpha t)$) in a time interval of length less
than $t$
(see Lemma \ref{lem:chainecoloree}). This is unlikely if $\alpha$ is chosen big
enough and if $t$ goes to infinity.

\begin{lem}\label{lem:chainecoloree}
Let $\varphi,\varphi' \in \Gamma$
agreeing on $\mathcal{B}(0,R+\alpha t)$ and $\omega, \omega' \in \Omega$
agreeing
on $\mathcal{B}(0,R+\alpha t)\times [0,t]$. If
$h(\cdot,\cdot,\cdot,\varphi;\omega)$ and
$h(\cdot,\cdot,\cdot,\varphi';\omega')$ differ on $\mathcal{B}(x,R) \times
[0,t]$, then, there must exist a
sequence $(x_i,y_i,t_i)_{0 \leq i \leq k }$ with $k:=\lfloor \alpha t
\rfloor$
satisfying
\begin{itemize}
\item $|y_0| \leq R$ and $|y_{i+1}-y_{i}| \leq 1$, for all $i \in \llbracket 0, k-1 \rrbracket$,
\item $0 \leq t_k \leq \cdots \leq t_0 \leq t$ and $(x_i,t_i) \in \omega_{y_i}$ for all $i \in \llbracket 0, k \rrbracket$,
with $\omega_y$ as  in Definition \eqref{defi:omegaphi}
\item $|x_0| \leq R+t-t_0$ and $|x_{i+1}-x_{i}| \leq t_{i} - t_{i+1}$ for all $i \in \llbracket 0, k-1 \rrbracket$.
\end{itemize}
\end{lem}
Before proving this Lemma, let us finish the proof of Proposition
\ref{prop:prolinGW}. If $\omega \notin A_{0,R,t,\alpha}$, then there must
exists $\varphi,\varphi',\omega'$ as in Lemma \ref{lem:chainecoloree} and some
$u \in [0,t]$ such that $h(\cdot,\cdot,\cdot,\varphi;\tau_u \omega)$ and
$h(\cdot,\cdot,\cdot,\varphi';\tau_u \omega')$ differ on $\mathcal{B}(x,R)
\times [0,t-u]$. By applying Lemma \ref{lem:chainecoloree} at time $t-u$ and
with creations $\tau_u \omega$ and $\tau_u \omega'$ we get a chain of creations
$(x_i,y_i,t_i)_{0 \leq i \leq \lfloor \alpha t
	\rfloor }$ such that $(x_i,y_i,t_i+u)_{0 \leq i \leq \lfloor \alpha t
	\rfloor }$ satisfies the $3$ points in Lemma \ref{lem:chainecoloree}. In
	order to be consistent with the definition of $C^{\uparrow}$ in Appendix \ref{sec:lis}, we relabel this sequence by setting
	$(x'_i,y'_i,t'_i)_{0 \leq i \leq \lfloor \alpha t
		\rfloor } := (x_{\lfloor \alpha t
		\rfloor-i},y_{\lfloor \alpha t
		\rfloor-i},t_{\lfloor \alpha t
		\rfloor-i}+u)_{0 \leq i \leq \lfloor \alpha t
		\rfloor }$ so that $|x'_{i+1}-x'_{i}| \leq t'_{i+1} - t'_{i}$. Doing
		this, we see that
\[
^c\!A_{0,R,t,\alpha} \subseteq \bigcup_{\ubar{y}' \in \mathcal{Y}_{R,\lfloor
\alpha t \rfloor}}  C^{\uparrow}_{\omega,\ubar{y}'}(T_{R,t})
\]
where $\mathcal{Y}_{R,n}$ is defined by
$\mathcal{Y}_{R,n} := \{ (y'_0\cdots y'_n)\in \Z^{n+1}, \ |y'_n| \leq R,
\ \forall i \in \llbracket 0,n-1 \rrbracket \, |y'_{i+1}-y'_i|\leq 1 \}$
 and $T_{R,t}$ is the trapezoid defined by
$T_{R,t}:=\{(x,s), \ s \in [0,t], \ |x| \leq R+t-s\}$. By Corollary
\ref{coro:lis}, since $T_{R,t}$ is of vertical diameter $t$ and area
$2Rt + t^2$, for any
$\ubar{y} \in \mathcal{Y}_{R,\lfloor \alpha t \rfloor}$,
	\[
	\Pro\left(C^{\uparrow}_{\omega,\ubar{y}}(T_{R,t}) \right) \leq  2 \, (2Rt +
	t^2) \,
	\left(\frac{4e^2\, t^2 }{\lfloor \alpha t \rfloor^2}\right)^{\lfloor \alpha
	t \rfloor}.
      \]
	Now, since $\mathcal{Y}_{R,\lfloor \alpha t \rfloor}$ is of cardinality bounded by $(2
	\lfloor R \rfloor +1) \, 3^{\lfloor \alpha t \rfloor}$, by union bound,
		\begin{equation}\label{eq:unionlocal}
		\Pro\left(^c\!A_{0,R,t,\alpha} \right) \leq  2 \, (2Rt + t^2) \, (2
		\lfloor R \rfloor +1) \,
		\left(\frac{12e^2\, t^2 }{\lfloor \alpha t \rfloor^2}\right)^{\lfloor \alpha t \rfloor},
		\end{equation}
and thus, for $\alpha := \sqrt{24} \, e$ we get
\[
\Pro\left(^c\!A_{0,nR,nt,\alpha} \right) \stackrel{n \to \infty} = \bigo \left( n^3 \,
		2^{- \alpha n t } \right),
\]
and the proof of Proposition \ref{prop:prolinGW} is concluded by choosing any $\gamma < \alpha \ln 2$.
\end{proof}

\begin{proof}[Proof of Lemma \ref{lem:chainecoloree}]
Let us introduce some notations (we will also use the notation
$\omega^\varphi,\omega^{\varphi}_y $ as in Definition \ref{defi:omegaphi}).
For all $(x,t) \in \R \times \R_+$, we define
$C_{x,t}^- := \{ (z,s) \in \R \times \R_+, \: |z-x| \leq t -s\}$. By
speed one propagation of kinks/antikinks, $h(x,y,t,\varphi;\omega)$
only depends on $\omega_y^{\varphi} \cap C_{x,t}^-$ and on
$\varphi(z,y)$ for $z \in [x-t,x+t]$. This fact can also be seen as a
consequence of Lemma \ref{lem:varformPNG} below. Now, we are
going to construct by induction a chain of creations like in
Lemma \ref{lem:chainecoloree}.

\underline{Construction of $(x_0,y_0,t_0)$ :}  Assume that there exists
$(x,y,s)\in \mathcal{B}(0,R)\times[0,t]$ such that we have  $h(x,y,s,\varphi;\omega)\neq
h(x,y,s,\varphi';\omega')$. Let us fix such a $(x,y,s)$ and set $y_0:=y$. By the
discussion above, since
$\varphi(\cdot,y_0)$ and $\varphi'(\cdot,y_0)$ agree on the interval
$[-R-\alpha t , R + \alpha t] \supseteq [x-t,x+t]$ (because $|x| \leq R$ and
$\alpha > 1$), necessarily $\omega_{y_0}^{\varphi} \cap  C_{x,s}^-$ and
$(\omega')^{\varphi'}_{y_0} \cap C_{x,s}^-$ are distinct. In other words, we
can find $(x_0,t_0) \in \omega_{y_0} \cap C_{x,s}^-$ ($=\omega_{y_0}' \cap
C_{x,s}^-$ by assumption on $\omega'$) corresponding to a kink/antikink
creation that occurs
for one of the dynamics but not for both (and such that $|x_0| \leq R + t -
t_0$). Consequently, the height functions must differ either at
$(x_0,y_0-1,t_0^-)$, $(x_0,y_0,t_0^-)$  or $(x_0,y_0+1,t_0^-)$ (otherwise the creation would have
been accepted or rejected simultaneously in both dynamics).

 \underline{Construction of $(x_{i+1},y_{i+1},t_{i+1})$ :} According to the
 three possibilities above, we set $y_1$ to be equal to $y_0-1,y_0$ or $y_0+1$
 (respectively in the first, second and third possibility). If $y_1$ is still
 in $\mathcal{B}(0,R+\alpha t)$, we can repeat the procedure above and find
 some $(x_1,t_1) \in \omega_{y_1} \cap C_{x_0,t_0}^-$ (hence $|x_1-x_0| \leq
 t_0 - t_1$) corresponding to a creation that occurs for one of the dynamic but
 not for both and so on. This construction continues as long as $y_i \in
 \mathcal{B}(0,R+\alpha t)$ and note that $y_i$ cannot exit $\mathcal{B}(0,R+\alpha t)$ for $i \leq \lfloor
 \alpha t \rfloor$. Overall, we  constructed a sequence as in Lemma
 \ref{lem:chainecoloree}. Its length is at least $\lfloor \alpha t \rfloor
 +1$.
\end{proof}

Now, let us show a Lemma that relates the linear propagation of information
with a Lipschitz property with respect to the initial height profile.

\begin{lem}\label{lem:locality}
	For all $R\geq0$,
	all $s \leq t$, all $x\in
	\R^2$ and all $n \in \N$, the following event happens with probability $1-
	\bigo \left( e^{- \gamma  n}\right)$ as $n$ goes to infinity (with $\gamma$ as in Proposition \ref{prop:prolinGW}):
        \begin{eqnarray}
          \label{eq:cor3}
	 \sup_{z \in n \mathcal{B}(x,R)}|h(z,v-u;\varphi,\tau_u
	 \omega) -
	h(z,v-u;\varphi',\tau_u \omega)|  \leq  \sup_{z \in
	n\mathcal{B}(x,R+\alpha	(t-s))} |\varphi(z) - \varphi'(z)|,
        \end{eqnarray}
	for every $\varphi,\varphi' \in \Gamma$ and every $u,v$ such that  $ns  \leq u \leq v \leq nt$ (and with $\alpha$ as in Proposition~\ref{prop:prolinGW}).
      \end{lem}
\begin{proof}
	By time translation invariance of the law of the Poisson process, we can assume that $s=0$.
	We are going to show that the event
	$A_{nx,nR,nt,\alpha}$ is included in the
	event in the l.h.s. of  \eqref{eq:cor3}. To do this, let us fix $\omega \in
	A_{nx,nR,nt,\alpha}$, $\varphi,\varphi' \in \Gamma$ and $0 \leq u \leq v
	\leq nt$. We
	set $$m = \sup_{z \in
		n\mathcal{B}(x,R+\alpha t)}|\varphi(z) - \varphi'(z)| \in \N \ , $$
	and $\tilde{\varphi} := \varphi \vee (\varphi' + m)$. It is not hard to
	show
	that $\tilde{\varphi} \in \Gamma$. Now, for all $z \in n\mathcal{B}(x,R)$,
	\begin{align*}
h(z,v-u;\varphi,\tau_u\omega)  &\leq
	h(z,v-u;\tilde{\varphi},\tau_u\omega) && \text{by Lemma
		\ref{lem:monotoneGW} since $\varphi \leq \tilde{\varphi}$}\\
	& = h(z,v-u;\varphi' + m,\tau_u\omega) && \text{($\omega
	\in
		A_{nx,nR,nt,\alpha}$ and $\tilde{\varphi} =
		(\varphi' + m)$ on
		$n\mathcal{B}(x,R+\alpha
		t)$)
	}\\
	 & = h(z,v-u;\varphi',\tau_u\omega) + m && \text{by Lemma
	\ref{lem:transinv}}.
	\end{align*}
	We can prove the other inequality by exchanging $\varphi$ and $\varphi'$
	which concludes this proof.
\end{proof}

Let us conclude this section by the next corollary which will be very useful
later on.
\begin{coro}[Asymptotic locality]\label{coro:asymlocal}
	There exists $\alpha > 1$ and a subset $\Omega_0 \subseteq \Omega$ of probability $1$ such
	that for all $\omega \in
	\Omega_0$, $x \in \R^2$, $R\geq0$,
	$(s,t) \in \R^2$ with $0 \leq s \leq t$, there exists $N(\omega)\in
	\N$ such that
	for all $n \geq N(\omega)$ and all
	$\varphi,\varphi' \in \Gamma$:
	\begin{equation}\label{eq:asymplipGW}
	\begin{aligned}
	& \sup_{s\leq u\leq v\leq t}\left\|h\left(n\cdot, \lfloor n \cdot
	\rfloor,n(v-u);\varphi,\tau_{nu}\omega\right) -
	h\left(n\cdot,
	\lfloor n \cdot \rfloor,n(v-u);\varphi',\tau_{nu}\omega\right)
	\right\| _{\infty}^{\mathcal{B}(x,R)} \\
	& \hspace{8cm} \leq
	\|\varphi(n\cdot,\lfloor n\cdot \rfloor)-\varphi'(n\cdot,\lfloor
	n\cdot \rfloor)\|_{\infty}^{\mathcal{B}(x,R+\alpha(t-s))}.
	\end{aligned}
	\end{equation}
\end{coro}

The proof follows easily from
Lemma \ref{lem:locality}, Borel-Cantelli Lemma and rational approximation (up
to choosing an $\alpha$ slightly larger than the one in Lemma
\ref{lem:locality}).

\subsection{Definition of the sequence of random discrete
semi-groups}\label{sec:defopGW}

We are going to define a sequence of functions $(S_n(s,t,f,\omega))_{n\in \N}$
describing the rescaled dynamic,  between times
$s$ and $t$ and Poissonian creations $\omega$, starting at time $s$ from an initial height profile close
to a continuous function $f$.
\begin{defi}\label{defi:S_n}
	For all $\omega \in \Omega$ and $(s,t) \in \T :=\{ (s,t) \in [0,T]^2, \, s
	\leq t\}$, we define
	\begin{equation}
	S_n(s,t;.,\omega) :\left\{
	\begin{aligned}
	& \bar{\Gamma}  & \longrightarrow&  \F(\R^2) \\
	& f  & \longmapsto & \frac{1}{n} h(n\cdot,\lfloor n\cdot
	\rfloor,
	n(t-s);\varphi_n^f,\tau_{ns}\omega),
	\end{aligned}
	\right.
	\end{equation}
	where $\F(\R^2)$ denotes the
	set of functions from $\R^2$ to $\R$, $\tau_{ns}$ is the temporal
	translation defined in \eqref{eq:translation} and $\varphi_n^f$ is the
	height profile in $\Gamma$ approaching $f$ as in the following Proposition
	\ref{prop:defphiGW}.
\end{defi}

\begin{prop}\label{prop:defphiGW}
	For all $n \in \N$, there exists a mapping
	\begin{equation}
	\left\{
	\begin{aligned}
	\bar{\Gamma} &\longrightarrow& \Gamma \\
	f &\longmapsto& \varphi_n^f
	\end{aligned}
	\right.
	\end{equation}
	satisfying that for all $c \in \R$ and $f \in \bar{\Gamma}$,
	$\varphi_n^{f+n^{-1}\lfloor n c\rfloor} = \varphi_n^{f}
	+ \lfloor n c\rfloor$, and such that
	\begin{equation}\label{eq:approxphiGW}
	\sup_{x,y \in \R^2} \left|\frac{1}{n}\varphi_n^f(nx,\lfloor ny \rfloor) - f(x,y)\right|
	\leq \frac{2}{n},
	\end{equation}
	Therefore, the sequence of functions $(\varphi_n^f)_{n \in \N}$ approaches
	$f$ in the sense of \eqref{eq:hypcondinitialGW}.
\end{prop}

\begin{rem}
  We cannot just choose
  $\varphi_n^f := (x,y) \mapsto \lfloor n f(n^{-1}x,n^{-1}y) \rfloor$
  because it could possibly have an accumulation point of
  discontinuities if $f$ oscillates too much; this would violate the first
  condition  in Definition \ref{defi:statespace}.
\end{rem}

\begin{proof}
	For any fixed $y \in \Z$, we are going to define $\varphi_n^f(\cdot,y)$
as piecewise	constant  on $\R_+$ (we will construct it similarly on $\R_-$).
	Let us define
	inductively $X_0^y=0$, $\varphi_n^f(0,y) := \lfloor nf(0,y/n) \rfloor$
	and
		\[
		\begin{aligned}
			&X_{i+1}^y := \inf\{x \geq X_i^y, \, |n f(x/n,y/n)-
			\varphi_n^f(X_i^y,y)|
		\geq 1 \}  \qquad \text{(with $\inf \emptyset = +\infty$)}\\
		& \varphi_n^f(\cdot,y) =
		\varphi_n^f(X_{i}^y,y) \quad \text{on $(X_i^y,X_{i+1}^y)$}\\
		& \varphi_n^f(X_{i+1}^y,y) =  n f(X_{i+1}^y/n,y/n).
		\end{aligned}
	\]
	By induction and by continuity of $f$, $\varphi_n^f(X_{i}^y,y) \in \Z$ for
	all $i$. Still by continuity, $X_{i+1}^y>X_i^y$ and
	$\{X_i^y, \, i \in \N\}$ is a locally finite subset of $\R_+$ with
	$\lim_{i \to \infty}X_i^y
	= +\infty$. Similarly, we construct $\varphi_n^f(\cdot,y)$ on negative real
	numbers. Up to
	modifying the value at discontinuity points, we
	obtain a function $\varphi_n^f(\cdot,y)$ which satisfies point 1 of
	Definition
	\ref{defi:statespace} and which satisfies the translation invariance
	property $\varphi_n^{f+n^{-1}c} = \varphi_n^{f}
	+ c$ for all $c \in \Z$ by construction.
	Moreover, by construction, for all $(x,y)
	\in \R \times \Z$,
	$|\varphi_n^f(x,y) -
	nf(x/n,y/n)| \leq 1$ and thus
	\begin{align*}
	&\sup_{(x,y) \in \R^2}|\frac{1}{n}\varphi_n^f(nx,\lfloor ny \rfloor) -
	f(x,y)| \\
	& \hspace{3cm} \leq \sup_{(x,y) \in \R^2}|\frac{1}{n}\varphi_n^f(nx,\lfloor
	ny
	\rfloor) -
	f(x,\lfloor ny \rfloor/n)| + \sup_{(x,y) \in \R^2}|f(x,\lfloor ny
	\rfloor/n)-f(x,y)| \\
	& \hspace{3cm} \leq \frac{1}{n} + |\lfloor ny \rfloor/n-y| \leq \frac{2}{n} \qquad
	\text{because
		$f \in
		\bar{\Gamma}$}.
	\end{align*}
	It remains to check that $\varphi_n^f$ satisfies point 2 of Definition
	\ref{defi:statespace} and hence is  in
	${\Gamma}$. To do this, let us fix $y \in \Z$ and show that for all
	$x\geq0$, $\varphi_n^{f}(x,y+1)-\varphi_n^{f}(x,y) \in \{-1,0\}$ (the case
	$x<0$ being similar). Let $x \geq 0$ and  $i,j$ be the unique integers such that $X_i^y \leq x < X_{i+1}^y$ and $X_j^{y+1} \leq x < X_{j+1}^{y+1}$. By construction of $\varphi_n^f$,
	\begin{align*}
	\varphi_n^f(x,y) &=  n
	f(X_i^y/n,y/n) \\
	\varphi_n^f(x,y+1) &=  n
	f(X_j^{y+1}/n,(y+1)/n) \ .
	\end{align*}
	There are two cases: either $X_j^{y+1} \in [X_i^y,X_{i+1}^{y})$ or $X_i^{y} \in [X_j^{y+1},X_{j+1}^{y+1})$. Since they are similar, we will just treat the first one. By definition of $X_{i+1}^y$, for all $z \in  [X_i^y, X_{i+1}^y)$ we have
	$|nf(z/n,y/n)-nf(X_i^y/n,y/n)|<1$ and thus
	\begin{align*}
	& \varphi_n^f(x,y+1) - \varphi_n^f(x,y)  =n f(X_j^{y+1}/n,(y+1)/n) -
n	f(X_i^y/n,y/n)  \\
	& =  \underbrace{nf(X_j^{y+1}/n,(y+1)/n) -
	nf(X_j^{y+1}/n,y/n)}_{\in[-1,0] \text{ since }f \in \bar{\Gamma}} +
	\underbrace{nf(X_j^{y+1}/n,y/n) - nf(X_i^y/n,y/n)}_{\in (-1,1) \text{ since
	} X_j^{y+1} \in [X_i^y,X_{i+1}^{y}) }.
	\end{align*}
	Finally, $\varphi_n^f(x,y+1) - \varphi_n^f(x,y) \in (-2,1)\cap \Z =
	\{-1,0\}$.
\end{proof}

\section{Compactness}

\subsection{Control on spatio-temporal height differences}\label{sec:control}

In this section, we control the spatio-temporal gradients of the
height function following the Gates-Westcott dynamic by comparison
with the PNG dynamic.  By construction,
$(x,t)\mapsto h(x,y,t;\varphi,\omega)$ follows the PNG dynamic (see
e.g \cite[Section 2]{ferrari2005one} for an introduction to the model)
starting from initial condition $\varphi(\cdot,y)$ with creation
locations given by~$\omega^{\varphi}_y$ as in Definition \ref{defi:omegaphi}.
This simple remark allows us
to use the representation of PNG model in terms of directed polymer on
Poisson points (see \cite[Section 3.1]{ferrari2005one}). First, we need
to introduce some new definitions.

\begin{defi}\label{defi:light}
For any finite set $A \subseteq
\R^2$, we define $L^{\uparrow}(A)$ as the maximal number of
points in $A$ that can be collected by a
\emph{light-path} i.e a continuous path $\gamma : [0,1] \to \R^2$ satisfying that
for any $0 \leq a \leq b$, we have $\gamma(b) - \gamma(a) \in \{(x,t)\in \R^2 \ |x|
\leq  t \}$.

We say that a rectangle $R \subseteq \R^2$ is a \emph{light-rectangle} if its
sides are
	parallel to the straight lines ${t=x}$ or ${t=-x}$. For any $s<t$ and
	$(x,s),(y,t)$ such
	that $|y-x| \leq t-s$, we note $R_{(x,s),(y,t)}$ the unique
	light-rectangle of diagonal $[(x,s),(y,t)]$.
\end{defi}

\begin{rem}\label{rem:lightrectangle}
We let the reader check that the area of $R_{(x,s),(y,t)}$ is $((t-s)^2-(y-x)^2)/2$ and that if $|x'-x|\leq s-s'$ then
	$R_{(x,s),(y,t)} \subseteq R_{(x',s'),(y,t)}$ while if $|y'-y|\leq
	t'-t$ then $R_{(x,s),(y,t)} \subseteq R_{(x,s),(y',t')}$.
\end{rem}

The next Lemma is an easy extension to arbitrary initial conditions of the
equivalence between the PNG and the directed polymers model as explained in
\cite[Section 2.3 and 3.1]{ferrari2005one} for special ``droplet'' and ``flat''
initial conditions (see also Figure \ref{fig:rep_graph}).

\begin{lem}\label{lem:varformPNG}
	For all $(x,y,t) \in \R \times \Z \times \R_+$, all $\varphi \in \Gamma$
	and $\omega \in \Omega$,
	\begin{equation}\label{eq:varform}
	h(x,y,t;\varphi,\omega) = \sup_{z \in [x-t,x+t]}\{\varphi(z,y) +
	L^{\uparrow}(\omega^{\varphi}_y \cap R_{(z,0),(x,t)}) \},
	\end{equation}
	and the supremum is attained for some $z\in[x-t,x+t]$.
\end{lem}

\begin{figure}[htbp]
	\begin{center}
		\includegraphics[scale=0.8]{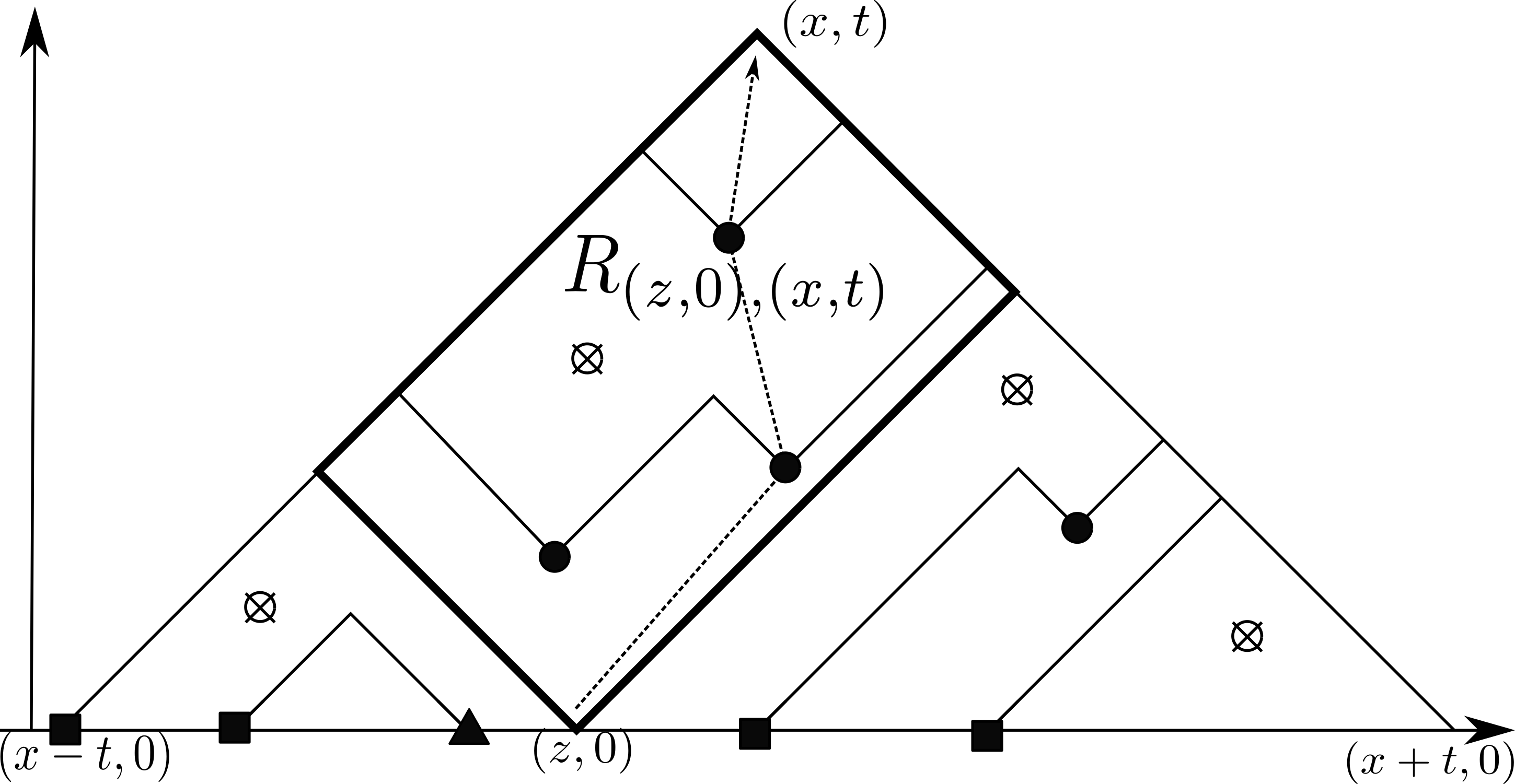}
	\end{center}
	\vspace{-0.5cm}
	\caption{\footnotesize The graphical representation of the PNG and its
	interpretation in terms of
	Directed Polymer. The lines are drawn from the initial antikinks (triangles)
	and kinks (squares) of $\varphi(\cdot,y)$
	and the effective creations in $\omega^\varphi_y$ (filled circles) (the
	creations in $\omega_y\setminus \omega_y^\varphi$ are marked by crossed
	circles and can be ignored). The height difference
	between two points is equal to the number of lines crossed by any
	light-path joining these points. A path touches at most one effective
	creation per line crossed and conversely, we can always find a path passing
	by effective creations that realises the
	maximum in the variational formula of  Lemma \ref{lem:varformPNG}.}
	\label{fig:rep_graph}
\end{figure}

In order to control the space gradients of the interface, we need an upper
bound
on $L^{\uparrow}(\omega_y^{\varphi} \cap R)$ (or on
$L^{\uparrow}(\omega_y \cap R)$ since $\omega_y^{\varphi} \subseteq
\omega_y$) for large rectangles $R$. This quantity is
well studied as it is related to the length of the longest increasing
subsequence of
a random uniform permutation, which was
shown first by Hammersley to behave like the square root of the number
of Poisson points in $R$ (this is also known as Ulam's problem;
see \cite{hammersley1972few}).
\begin{lem}\label{lem:lis}
	There
	exists a constant $c>0$ such that for all $\omega$ in a subset of $\Omega$
	of probability $1$, for
	all
	light-rectangle $R\subseteq \R^2$ and all $Y>0$,
	\begin{equation}\label{eq:lis}
	\limsup_{n \to \infty} \sup_{y \in \llbracket -nY,nY
	\rrbracket}\frac{1}{n}L^{\uparrow}(\omega_{y} \cap nR)
	\leq c \, \sqrt{\Leb(R)},
	\end{equation}
	where $\Leb(R)$ is the area of $R$. Therefore, up to intersecting this
	subset of
	probability $1$ with $\Omega_0$ (defined in Corollary \ref{coro:asymlocal})
	we can assume that \eqref{eq:lis} holds for all $\omega \in \Omega_0$.
\end{lem}
\begin{proof}
	By Lemma \ref{lem:lisr} in Appendix \ref{sec:lis}, for all $y\in \Z$ and
	all
	$k \in \N$,
	\[
	\Pro\left(L^{\uparrow}(\omega_{y} \cap R) \geq k \right) \leq
	\left(\frac{2e^2\,  \Leb(R)}{k^2}\right)^k.
	\]
	This is a classical inequality when dealing with longest increasing subsequences that can be found for example in \cite[Lemma 4.1]{seppalainen1996microscopic}. Therefore, for $c=2e$,
	\[
	\Pro\left(\sup_{y \in \llbracket -nY,nY \rrbracket}
	\left\{\frac{1}{n}L^{\uparrow}(\omega_{y} \cap nR)\right \} \geq  c
	\sqrt{\Leb(R)} \right) \leq  2 n Y \, 2^{- \lceil n c\sqrt{\Leb(R)}
	\rceil }.
	\]
	By Borel-Cantelli Lemma, for almost all $\omega$, for any $Y>0$,
	\[
	\limsup_{n \to \infty} \sup_{y \in \llbracket -nY,nY
	\rrbracket}\frac{1}{n}L^{\uparrow}(\omega_{y} \cap nR)
	\leq c \, \sqrt{\Leb(R)}.
	\]
	By countability, we can have this almost surely simultaneously for all
	light-rectangle $R$ with rational coordinate vertices. The full proof
	follows
	by density of rational numbers and by the monotonicity with respect to inclusion of
	$R \mapsto
	L^{\uparrow}(\omega_{y} \cap nR)$.
\end{proof}
Let us now give some consequences of Lemmas \ref{lem:varformPNG} and
\ref{lem:lis}.

\begin{prop}[Control on temporal growth]\label{prop:controlsgrowth}
	For all $\omega \in \Omega_0$, for all $f \in \bar{\Gamma}$, all $s
	\leq t$ and $(x,y)\in \R^2$,
	\begin{equation}
	\limsup_{n \to \infty} S_n(s,t,f;\omega)(x,y) \leq \sup_{|z-x|\leq
	t-s}f(z,y)  + \sqrt{2} \, c \ (t-s),
	\end{equation}
	where $c$ is the same constant as in Lemma \ref{lem:lis}.
\end{prop}
\begin{proof}
	By definition of $S_n$ and by Lemma \ref{lem:varformPNG},
	\begin{align*}
	&S_n(s,t,f;\omega)(x,y) = \frac{1}{n} h(nx,\lfloor ny
	\rfloor,
	n(t-s);\varphi_n^f,\tau_{ns}\omega)\\
	& = \sup_{|z-x|\leq t-s}\left\{\frac{1}{n}\varphi_n^f(nz,y) +
	\frac{1}{n}L^{\uparrow}\left( (\tau_{ns}\omega)^{\varphi_n^f}_{\lfloor ny
		\rfloor} \cap R_{(nz,0),(nx,n(t-s))}\right) \right\}\\
	& \leq \sup_{|z-x|\leq t-s}\left\{f(z,y)+\frac{2}{n} +
	\frac{1}{n}L^{\uparrow}\left( (\tau_{ns} \omega)_{\lfloor ny \rfloor} \cap
	nR_{(z,0),(x,t-s)}\right) \right\} \qquad \text{by \eqref{eq:approxphiGW} and
	since $\omega^{\varphi_n^f} \subseteq \omega$}\\
	& = \sup_{|z-x|\leq t-s}\left\{f(z,y)+\frac{2}{n} +
	\frac{1}{n}L^{\uparrow}\left(  \omega_{\lfloor ny \rfloor} \cap
	nR_{(z,s),(x,s)}\right) \right\} \qquad \text{by definition of $\tau_{ns}$ in \eqref{eq:translation}}\\
	& \leq \sup_{|z-x|\leq t-s}f(z,y)+\frac{2}{n} +
	\frac{1}{n}L^{\uparrow}\left(\omega_{\lfloor ny
		\rfloor} \cap nR_{(x,2s-t),(x,t)}\right) \qquad \text{ since
		$R_{(z,s),(x,t)} \subseteq R_{(x,2s-t),(x,t)}$}.
	\end{align*}
	We conclude the proof with Lemma \ref{lem:lis} and
	$\Leb(R_{(x,2s-t),(x,t)}) = 2 (t-s)^2$.
\end{proof}

Now, we establish a crucial lemma that guarantees a priori that, at any time,
the asymptotic rescaled height function has at least the worst regularity
between
that of the initial height
profile and $1/2$-Hölder regularity. A posteriori, after the proof
of the main
theorem, we will have that it stays Lipschitz at any time if the initial
condition is itself Lipschitz since this is the case for viscosity solutions of
Hamilton-Jacobi equations.

\begin{prop}[Control on height differences along $x$]\label{prop:difhGW}
	There exists a constant $C$ (that depends on the time horizon $T$) such that for all
	$\omega\in \Omega_0$
	, all
	$f
	\in \bar{\Gamma}$, all $(x,y) \in
	\R^2$ , and all $\delta \in [0,1]$,
	\begin{equation}
	\begin{aligned}
	\limsup_{n \to \infty} \sup_{\substack{x_1,x_2 \in [x-\delta,x+\delta] \\
			0 \leq s \leq t \leq T}} &
	|S_n(s,t,f;\omega)(x_2,y)-S_n(s,t,f;\omega)(x_1,y)| \\
	& \leq
	\sup_{\substack{x_1,x_2\in[x-\delta-T,x+\delta+T] \\ |x_2-x_1| \leq
			2 \delta }} |f(x_1,y)-f(x_2,y)| +
	C \, \sqrt{\delta} .
	\end{aligned}
	\end{equation}
\end{prop}
\begin{proof}
	We start by showing the following Lemma.
	\begin{lem}\label{lem:controlrect}
		For all $y \in \Z$, all $x_1<x_2 \in \R$, all $t \geq 0$, all
		$\varphi \in \Gamma$,
		\begin{equation}
		\begin{aligned}
		|h(x_2,y,t;\varphi,\omega)&-h(x_1,y,t;\varphi,\omega)|
		\leq
		\sup_{\substack{z,z'\in[x_1-t,x_2+t] \\ |z-z'|\leq |x_2-x_1|}}
		|\varphi(z,y)-\varphi(z',y)|+ \max \left(L^{\uparrow}\left(\omega_y
		\cap R_1 \right),L^{\uparrow}\left(\omega_y \cap R_2 \right) \right),
		\end{aligned}
		\end{equation}
		with $R_1 := R_{\big(\frac{x_1+x_2}{2}-t,-\frac{x_2-x_1}{2}\big),\big(x_1,t\big)}$ and $R_2:=		R_{\big(\frac{x_1+x_2}{2}+t,-\frac{x_2-x_1}{2}\big),\big(x_2,t\big)}$.
	\end{lem}
        \begin{proof}
	We start by showing that
	\begin{equation}\label{eq:controlrectangle}
	h(x_2,y,t;\varphi,\omega)-h(x_1,y,t;\varphi,\omega) \leq
	\sup_{z\in[x_1+t,x_2+t]}
	|\varphi(z,y)-\varphi(x_1+t,y)| + L^{\uparrow}\left(\omega_y \cap
			R_2\right).
	\end{equation}
	By Lemma \ref{lem:varformPNG}, there exists $z \in [x_2-t,x_2+t]$ such that
	\[h(x_2,y,t;\varphi,\omega) = \varphi(z,y) +
	L^{\uparrow}(\omega^{\varphi}_y \cap R_{(z,0),(x_2,t)}) \ .
	\]
	Two cases can occur:
	\begin{enumerate}
		\item [(i)]	 If $z \in
		[x_1-t,x_1+t]$, then, by Lemma \ref{lem:varformPNG},
		$h(x_1,y,t;\varphi,\omega)
		\geq \varphi(z,y) +
		L^{\uparrow}(\omega^{\varphi}_y \cap R_{(z,0),(x_1,t)})$ hence
		\[
		h(x_2,y,t;\varphi,\omega)-h(x_1,y,t;\varphi,\omega) \leq
		L^{\uparrow}(\omega^{\varphi}_y \cap R_{(z,0),(x_2,t)}) -
		L^{\uparrow}(\omega^{\varphi}_y \cap R_{(z,0),(x_1,t)}).
		\]

		Now, for any $A,B
		\subseteq \R^2$ finite sets, it is not hard to show
		$
		L^{\uparrow}(A\cup B) \leq L^{\uparrow}(A) + L^{\uparrow}(B),
		$
		and thus for any $A,B
				\subseteq \R^2$ finite sets,
		\[
		L^{\uparrow}(A) - L^{\uparrow}(B) \leq L^{\uparrow}(A \setminus B).
		\]
		We apply this inequality with $A=\omega^{\varphi}_y \cap R_{(z,0),(x_2,t)}$ and $B=\omega^{\varphi}_y \cap R_{(z,0),(x_1,t)}$ (the creations inside the blue and red rectangles on Figure \ref{fig:control}). The set $A\setminus B$ is equal to the creations inside the green rectangle which is included in the light-rectangle $R_2$ (surrounded by dash lines on Figure \ref{fig:control}).

		\begin{figure}[htbp]
			\begin{center}
				\includegraphics[scale=0.8]{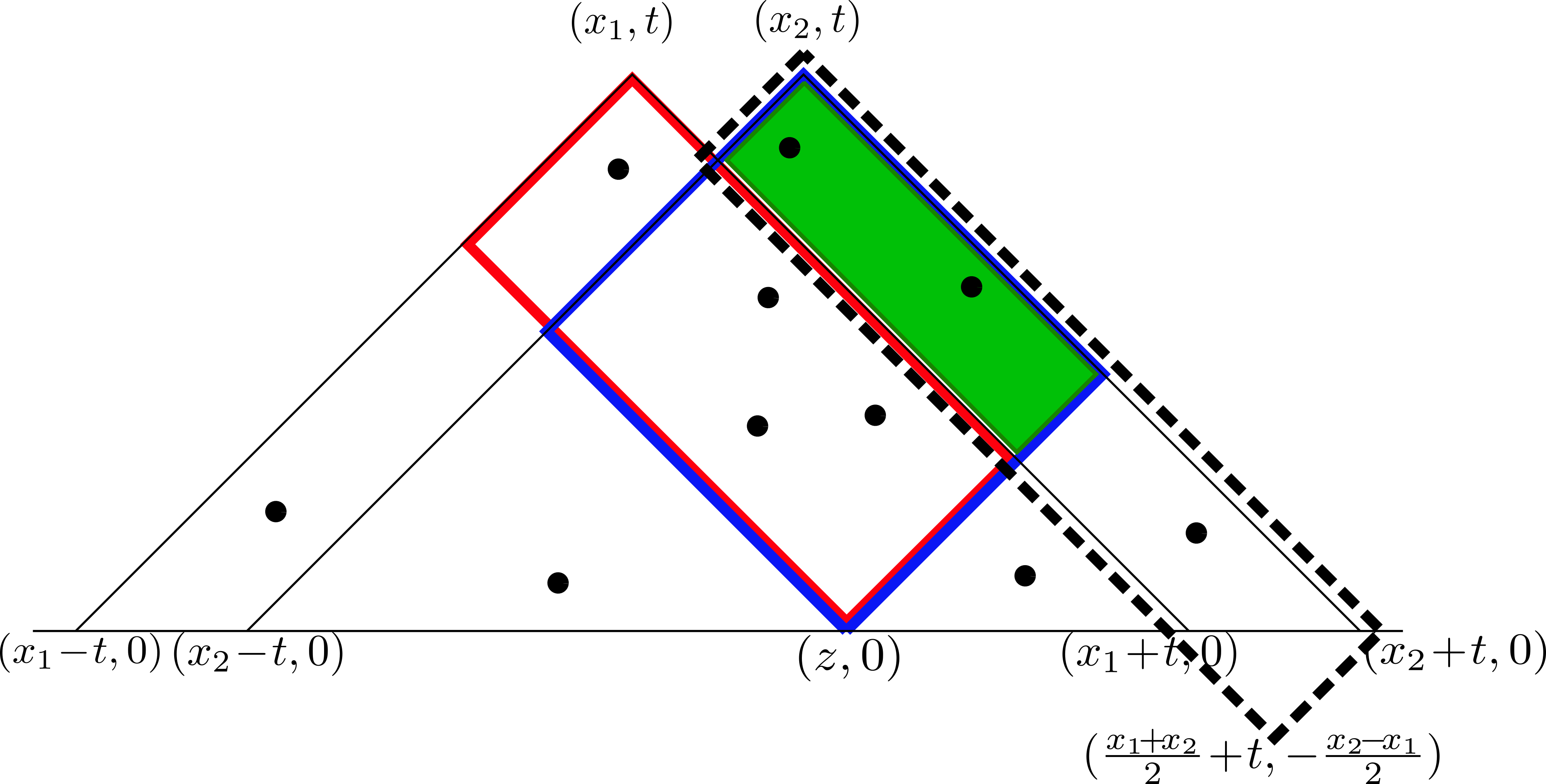}
			\end{center}
			\vspace{-0.5cm}
			\caption{\footnotesize Illustration of the proof of Lemma \ref{lem:controlrect}. The effective creations are depicted by black circles. The blue rectangle $R_{(z,0),(x_2,t)}$ and the red rectangle $R_{(z,0),(x_1,t)}$ are involved in the variational formula \eqref{eq:varform} applied to $h(x_2,y,t)$ and $h(x_1,y,t)$. The green rectangle corresponds to the set difference of the blue and red rectangles and is included in the larger dashed-line rectangle.}
			\label{fig:control}
		\end{figure}

		Altogether, we obtain $h(x_2,y,t;\varphi,\omega)-h(x_1,y,t;\varphi,\omega) \leq
		L^{\uparrow}\left(\omega^{\varphi}_y \cap R_2\right)$,
		which implies \eqref{eq:controlrectangle}.

		\item [(ii)]
		If $z \in
		[x_1+t,x_2+t]$, then by choosing $z' = x_1+t$ in the variational formula
		\eqref{eq:varform},
		we get that $h(x_1,y,t;\varphi,\omega) \geq \varphi(x_1+t,y)$ hence
		\begin{align*}
		& \hspace{-2cm} h(x_2,y,t;\varphi,\omega)-h(x_1,y,t;\varphi,\omega)\\
		& \leq
		\varphi(z,y) + L^{\uparrow}(\omega^{\varphi}_y \cap R_{(z,0),(x_2,t)})
		-
		\varphi(x_1+t,y)\\
		& \leq \sup_{z\in[x_1+t,x_2+t]}
		|\varphi(z,y)-\varphi(x_1+t,y)| +
		L^{\uparrow}\left(\omega_y^{\varphi} \cap R_2\right) \ ,
		\end{align*}
		since $$R_{(z,0),(x_2,t)} \subseteq
		R_{\big(\frac{x_1+x_2}{2}+t,-\frac{x_2-x_1}{2}\big),\big(x_2,t\big)} = R_2 \quad \text{for all $z \in [x_1+t,x_2+t]$},$$
		as shown in Figure \ref{fig:control}.	This also implies \eqref{eq:controlrectangle}.

	\end{enumerate}
	The proof of Lemma \ref{lem:controlrect} is concluded by showing similarly that
	$$h(x_1,y,t;\varphi,\omega)-h(x_2,y,t;\varphi,\omega) \leq
	\sup_{z\in[x_1-t,x_2-t]}
	|\varphi(z,y)-\varphi(x_2-t,y)| + L^{\uparrow}\left(\omega_y \cap
			R_1 \right).$$
\end{proof}
Lemma \ref{lem:controlrect} yields that for all $x-\delta \leq x_1 \leq x_2 \leq
x+
\delta$ and all $t \geq 0$,
\[
\begin{aligned}
& |h(x_2,y,t;\varphi,\omega)-h(x_1,y,t;\varphi,\omega)|\\
& \leq
\sup_{\substack{z,z'\in[x-\delta-t,x+\delta+t] \\ |z-z'|\leq 2\delta}}
|\varphi(z,y)-\varphi(z',y)|
+ \max \left(L^{\uparrow}(\omega_y \cap
R_{(x-t,-\delta),(x,t+\delta)}),L^{\uparrow}(\omega_y \cap
R_{(x+t,-\delta),(x,t+\delta)}) \right)
\end{aligned}
\]
since one can check that  $$R_{\big(\frac{x_1+x_2}{2}-t,-\frac{x_2-x_1}{2}\big),\big( x_1,t \big)} \subseteq R_{(x-
	t,-\delta),(x,t+\delta)}$$ and
$$R_{\big(\frac{x_1+x_2}{2}+t,-\frac{x_2-x_1}{2}\big),\big( x_2,t \big)} \subseteq R_{(x+
	t,-\delta),(x,t+\delta)}.$$ Therefore, for all
$x_1,x_2 \in [x-\delta,x+\delta]$, all
$s,t$ such that $0\leq s \leq t \leq T$, all $f \in \bar{\Gamma}$ and
all $n \in \N$,
\begin{align*}
|S_n(s,t,f;\omega)(x_2,y)& -S_n(s,t,f;\omega)(x_1,y)| \\
 = & \left|\frac{1}{n}h(nx_2,\lfloor ny
\rfloor,n(t-s),\varphi_n^f,\tau_{ns}\omega) - \frac{1}{n}h(nx_1,\lfloor ny
\rfloor,n(t-s),\varphi_n^f,\tau_{ns}\omega) \right| \\
 \leq & \sup_{\substack{z,z'\in[x-\delta-(t-s),x+\delta+(t-s)] \\ |z-z'|\leq 2\delta}}
\left| \frac{1}{n} \varphi_n^f(nz,\lfloor ny \rfloor)-\frac{1}{n}
\varphi_n^f(nz',\lfloor ny \rfloor)\right|\\
& + \max \left( \frac{1}{n}L^{\uparrow}(\omega_{\lfloor ny \rfloor} \cap
nR_{(x-(t-s),s-\delta),(x,t+\delta)}),\frac{1}{n}L^{\uparrow}(\omega_{\lfloor
ny \rfloor} \cap
nR_{(x+(t-s),s-\delta),(x,t+\delta)}) \right)\\
 \leq & \sup_{\substack{z,z'\in[x-\delta-T,x+\delta+T] \\ |z-z'|\leq 2\delta}}
\left| f(z,y)-f(z',y)\right| + 4/n\\
& + \max \left( \frac{1}{n}L^{\uparrow}(\omega_{\lfloor ny \rfloor} \cap
nR_{(x-T,-\delta),(x,T+\delta)}),\frac{1}{n}L^{\uparrow}(\omega_{\lfloor
ny \rfloor} \cap
nR_{(x+T,-\delta),(x,T+\delta)}) \right),
\end{align*}
where the last inequality holds because of \eqref{eq:approxphiGW} and $R_{(x\pm(t-s),s-\delta),(x,t+\delta)} \subseteq R_{(x\pm
	T,-\delta),(x,T+\delta)}$. Note that this upper bound is uniform in  $x_1,x_2 \in [x-\delta,x+\delta]$ and in
	$s,t$ such that $0 \leq s \leq t \leq T$. We conclude the proof of Proposition \ref{prop:difhGW} by applying Lemma
\ref{lem:lis} to $R_{(x\pm
	T,-\delta),(x,T+\delta)}$ which are of area  $2(\delta T + \delta^2) \leq 2
	(T+1) \, \delta$ since $\delta \in [0,1]$ and by choosing $C := c \,
\sqrt{2(T+1)}$ ($c$ is the same constant as in Lemma \ref{lem:lis}).
\end{proof}

\subsection{Choice of the metric}

We endow $\F(\R^2)$ and $\bar{\Gamma} \subseteq \mathcal{C}(\R^2)
\subseteq \F(\R^2)$ with the distance of uniform convergence on all compacts,
e.g.
\begin{equation}
\forall f_1,f_2 \in \F(\R^2), \; d_{\infty,c}(f_1,f_2) :=
\sum_{i=1}^{\infty}2^{-i} \left(\|f_1-f_2\|_{\infty}^{[-i,i]^2} \wedge 1\right).
\end{equation}
For this distance, a sequence of functions $(f_n)_{n \in \N}$ converges to $f
\in \F(\R^2)$ if and only if it converges uniformly on all compact sets of
$\R^2$ to $f$.

\begin{prop}\label{prop:metrique}
	The metric space $(\F(\R^2),d_{\infty,c})$ is complete. The metric space
	$(\bar{\Gamma},d_{\infty,c})$ is complete
	and separable (i.e. a Polish space).
\end{prop}

\begin{proof}
  The completeness of $(\F(\R^2),d_{\infty,c})$ is a classical
  fact. Since $\bar{\Gamma}$ is a closed subset of $\mathcal{C}(\R^2)$
  (which is complete because closed in $\F(\R^2)$ and separable by
  approximation by polynomials with rational coefficients on any
  compact set) it is in turn a complete separable metric space.
\end{proof}

Now, we denote $\F_{\T} := \F(\T,\F(\R^2))$ the set of functions from
$\T$ (defined in Definition \ref{defi:S_n}) into $\F(\R^2)$ which we endow with the uniform distance:
\begin{equation}
\forall F_1,F_2 \in \F_{\T}, \; D_{\infty}(F_1,F_2) := \sup_{0 \leq s \leq t
	\leq T}d_{\infty,c}(F_1(s,t),F_2(s,t)).
\end{equation}

The following Proposition is standard when dealing with functional spaces with a complete set of destination such as $\F(\R^2)$ (by Proposition \ref{prop:metrique}) and endowed with the uniform distance.

\begin{prop}\label{prop:metriquebis}
	\noindent The metric space $(\F_\T,D_{\infty})$ is complete.
\end{prop}

\subsection{Compactness for any fixed $\omega$ in a subset $\Omega_0$ of
probability $1$}

We recall that $\Omega_0$ is a subset of $\Omega$ of probability $1$
introduced in Corollary \ref{coro:asymlocal} and Lemma
  \ref{lem:lis}. The goal of this section is to show the following
proposition:
\begin{prop}\label{prop:compacity}
  For all $\omega \in \Omega_0$, and all sub-sequences
  $(n_k)_{k \in \N}$, we can extract a sub-sub-sequence
  $(n_{k_l})_{l \in \N}$ such that for all functions
  $f \in \bar{\Gamma}$, the sequence
  $(S_{n_{k_l}}(.,.;f,\omega))_{l\in \N}$ converges to a certain
  $S(\cdot,\cdot;f,\omega)$ in $\F_{\T}$, i.e,
	\[
	\forall R>0, \qquad \sup_{\substack{0 \leq s \leq t \leq T \\ |x|,|y| \leq
	R }}
	\left|
	S_{n_{k_l}}(s,t,f;\omega)(x,y)-S(s,t,f;\omega)(x,y) \right| \underset{l \to
		\infty}{\longrightarrow} 0.
	\]
	 Moreover, for all $(s,t) \in \T$,
	$f \mapsto  S(s,t;f,\omega)$ is continuous from $\bar{\Gamma}$ into itself
	and for all $f \in \bar{\Gamma}$,	$(s,t) \mapsto S(s,t;f,\omega)$ is
	continuous.
\end{prop}

\begin{proof}
	In all this proof, we fix $\omega \in \Omega_0$. Let us apply Proposition
	\ref{prop:ascoli} to the sequence of functions $\bar{\Gamma} \ni f
	\longmapsto S_n(\cdot,\cdot,f;\omega) \in \F_{\T}$. From Propositions
	\ref{prop:metrique} and \ref{prop:metriquebis},
	$\bar{\Gamma}$ is separable and $\F_{\T}$ is complete. Therefore, the proof
	of Proposition \ref{prop:compacity} follows easily from Proposition
	\ref{prop:ascoli} together with the next two lemmas giving asymptotic
	equi-continuity and pointwise relative compactness.

	\begin{lem}[Asymptotic equi-continuity of $(f \mapsto
		S_n(\cdot,\cdot,f;\omega))_{n \in \N}$]\label{lem:asymplipschitz}
		For all $\omega\in \Omega_0$ and all $\eps>0$, there exists $N \in
		\N$ such that
		\begin{equation}
		\forall n \geq N \quad \forall f,g\in \bar{\Gamma} \qquad D_{\infty}
		\left(S_n(\cdot,\cdot,f;\omega),S_n(\cdot,\cdot,g;\omega)\right) \leq
		2^{\lceil \alpha T \rceil} \,
		d_{\infty,c}(f,g) + \eps.
              \end{equation}
	\end{lem}
\noindent The proof comes from an easy corollary of \eqref{eq:asymplipGW} (we
will prove
it in details at
the end of
this section).

	\begin{lem}[Pointwise relative compactness of $((s,t) \mapsto
		S_n(s,t,f;\omega))_{n \in \N}$ in $\F_T$]\label{lem:prc1}
		For any $\omega \in \Omega_0$ and $f \in \bar{\Gamma}$, the sequence
		$((s,t) \mapsto S_n(s,t,f;\omega))_{n \in \N}$ is contained in a
		compact set
		of
		$\F_{\T}$. Moreover, any limit point is continuous
		from~$\T$
		into $\bar{\Gamma}$.
	\end{lem}

\begin{proof}[Proof of Lemma \ref{lem:prc1}]
We want to show that for any fixed $f \in \bar{\Gamma}$, from any sub-sequence
of $((s,t) \mapsto S_n(s,t,f;\omega))_{n \in \N}$, we can find a uniformly
converging
sub-sub-sequence in $\F_{\T}$. We are going to apply once again Proposition
\ref{prop:ascoli}. The set $\F_{\T}$ is the set of functions from $\T$
which is
compact into $(\F(\R^2),d_{\infty,c})$ which is complete (by Proposition
\ref{prop:metrique}). Therefore, it is
enough to show asymptotic equi-continuity and pointwise relative compactness.

\begin{lem}[Pointwise relative compactness of $(
	S_n(s,t,f;\omega))_{n \in \N}$ in $\F(\R^2)$]\label{lem:prc2}
	For any $\omega \in \Omega_0$, $f \in \bar{\Gamma}$ and $0 \leq s \leq t
	\leq T$, the sequence
	$(S_n(s,t,f;\omega))_{n \in \N}$ is contained in a
	compact set
	of
	$\F(\R^2)$. Moreover, any limit point is in $\bar{\Gamma}$.
\end{lem}
\begin{proof}[Proof of Lemma \ref{lem:prc2}]
	As $\F(\R^2)$ is endowed with the topology of convergence on all
	compact sets of $\R^2$,
	it is enough to show asymptotic
	equi-continuity and pointwise relative compactness in order to apply
	Proposition
	\ref{prop:ascoli} once more.
	\begin{enumerate}
		\item \emph{Pointwise relative compactness:
			$(S_n(s,t,f;\omega)(x,y))_{n \in
				\N}\in
			\R$}\\
		By Bolzano-Weierstrass Theorem, it is enough to show that this sequence
		is bounded. The upper bound is a
		direct
		consequence
		of Proposition \ref{prop:controlsgrowth} while the lower bound is
		trivial since
		height functions are non-decreasing with time.
		\item \emph{Asymptotic equi-continuity of $((x,y) \mapsto
			S_n(s,t,f;\omega)(x,y))_{n \in \N}$}\\
		Let $(x,y) \in \R^2$. By the slope constraint for functions in
		$\Gamma$, it is easy to check that for any $n\in \N$, $x \in \R$ and
		$y<y'$:
		\begin{equation}\label{eq:controlslopey}
		S_n(s,t,f;\omega)(x,y')-S_n(s,t,f;\omega)(x,y)\in
		\left[-(\lfloor ny' \rfloor - \lfloor ny \rfloor)/n
		,0\right].
		\end{equation}
		By this and Proposition
		\ref{prop:difhGW},
		\begin{equation}\label{eq:asympteq,x,y}
		\begin{aligned}
		\limsup_{n \to \infty} \sup_{\substack{(x',y') \in \R^2 \\
				|x-x'|,|y-y'| \leq
				\delta}} &|S_n(s,t,f;\omega)(x,y) - S_n(s,t,f;\omega)(x',y')|\\
		&  \leq \limsup_{n \to \infty} \sup_{ \substack{x' \in
				[x-\delta,x+\delta]}}
		|S_n(s,t,f;\omega)(x,y) - S_n(s,t,f;\omega)(x',y)| + \delta \\
		& \leq  \sup_{\substack{x_1,x_2\in[x-\delta-T,x+\delta+T]
				\\ |x_2-x_1| \leq
				2 \delta}} |f(x_1,y)-f(x_2,y)| +
		C \, \sqrt{\delta} + \delta \ .
		\end{aligned}
		\end{equation}
		By uniform continuity of $f$ on any compact, the right-hand side tends
		to $0$
		when $\delta$ tends to~$0$.
	\end{enumerate}

		Therefore, by Proposition \ref{prop:ascoli}, any subsequence of
		$(S_n(s,t,f;\omega))_{n
			\in \N}$ has a subsequence that converges in
			$(\F(\R^2),d_{\infty,c})$ and any limit
		point is continuous. Actually, by taking the limit in
		\eqref{eq:controlslopey}, any limit point is in $\bar{\Gamma}$. This
		concludes the proof of Lemma \ref{lem:prc2}.
\end{proof}

To finish the proof of Lemma \ref{lem:prc1}, we are going to show
asymptotic equi-continuity of
$((s,t) \mapsto S_n(s,t,f;\omega))_{n \in \N}$.  Let us fix
$\omega \in \Omega_0$, $f \in \bar{\Gamma}$ and $(s,t) \in \T$. By
definition of $d_{\infty,c}$, it is enough to show that for any
$\eps>0$ and $R>0$ there exists $\delta>0$ such that:
\[
\limsup_{n \to \infty} \sup_{\substack{(s',t') \in \T \\ |s'-s|\leq
		\delta, \,  |t'-t|\leq \delta}}
\| S_n(s,t,f;\omega)-S_n(s',t',f;\omega)\|_{\infty}^{[-R,R]^2} \leq \eps \ .
\]
We claim that for any $(s',t')\in \T$, there exists some $r \leq t
\wedge t'$ and $u \geq s \vee s'$ such that
\begin{equation}\label{eq:r,u}
\begin{aligned}
\| S_n&(s,t,f;\omega)-S_n(s',t',f;\omega)\|_{\infty}^{[-R,R]^2}\\
& \leq \|
S_n(r,t,f;\omega)-S_n(r,t',f;\omega)\|_{\infty}^{[-R,R]^2} + \|
S_n(s,u,f;\omega)-S_n(s',u,f;\omega)\|_{\infty}^{[-R,R]^2}.
\end{aligned}
\end{equation}
Indeed, at least one of the two conditions occurs:
$s \leq t'$ or $s'
\leq t$. In the first case, \eqref{eq:r,u} holds with $(r,u)=(s,t')$ while in
the second case, $(r,u)=(s',t)$. Therefore, asymptotic equi-continuity of
$((s,t)
\mapsto
S_n(s,t,f;\omega))_{n \in \N}$ follows from the next Lemma.
\begin{lem}
	For all $\omega \in \Omega_0$, $(s,t) \in \T$,  $f \in \bar{\Gamma}$, $R>0$
	and $\eps>0$, there exists $\delta>0$ such that
	\begin{equation}\label{eq:s,t,t'}
	\limsup_{n \to \infty} \sup_{\substack{r,t' \in [0,T] \\ r \leq t \wedge
			t',
			\,  |t'-t|
			\leq \delta}} \|
	S_n(r,t,f;\omega)-S_n(r,t',f;\omega)\|_{\infty}^{[-R,R]^2}
	\leq \eps
	\end{equation}
	and
	\begin{equation}\label{eq:s,s',t}
	\limsup_{n \to \infty} \sup_{\substack{u,s' \in [0,T] \\ u \geq s \vee s' ,
			\, |s'-s|
			\leq
			\delta}} \|
	S_n(s,u,f;\omega)-S_n(s',u,f;\omega)\|_{\infty}^{[-R,R]^2} \leq \eps.
	\end{equation}
\end{lem}

\begin{proof}
	We are going to prove
	\eqref{eq:s,t,t'} first. Let $(x,y) \in \R^2$ and suppose first  that $r\leq
	t \leq t'$. By
	Lemma~\ref{lem:markovGW},
	\begin{align*}
	S_n(r,t',f;\omega)(x,y) &= \frac{1}{n}h\left(nx,\lfloor ny \rfloor,
	n(t'-r),
	\varphi_n^f; \tau_{nr}\omega\right) \\
	& =\frac{1}{n}h\left(nx,\lfloor ny \rfloor, n(t'-t), h(n\cdot,\lfloor n
	\cdot \rfloor,n(t-r),\varphi_n^f,\tau_{nr}\omega); \tau_{nt}\omega\right).
	\end{align*}
	Now, by Lemma \ref{lem:varformPNG} applied with initial condition $\psi :=
	h(n\cdot,\lfloor n
	\cdot \rfloor,n(t-r),\varphi_n^f,\tau_{nr}\omega)$, we get
	\begin{align*}
	S_n(r,t',f;\omega)(x,y) &= \sup_{|z-x|\leq |t'-t|} \left\{
	\frac{1}{n}\psi(z,y) + \frac{1}{n}
	L^{\uparrow}\left((\tau_{nt}\omega)^{\psi}_{\lfloor ny \rfloor} \cap n
	R_{(z,0),(x,t'-t)}\right)  \right\} \\
	& \leq \sup_{|z-x|\leq |t'-t|} \left\{
	\frac{1}{n}\psi(z,y) + \frac{1}{n}
	L^{\uparrow}\left((\tau_{nt}\omega)_{\lfloor ny \rfloor} \cap n
	R_{(z,0),(x,t'-t)}\right)  \right\}\\
	& =  \sup_{|z-x|\leq |t'-t|} \left\{
	S_n(r,t,f;\omega)(z,y) + \frac{1}{n}
	L^{\uparrow}\left(\omega_{\lfloor ny \rfloor} \cap n
	R_{(z,t),(x,t')}\right)  \right\} \\
	& \leq \sup_{|z-x|\leq |t'-t|} S_n(r,t,f;\omega)(z,y) +\frac{1}{n}
	L^{\uparrow}\left(\omega_{\lfloor ny \rfloor} \cap n
	R_{(x,t-|t'-t|),(x,t')}\right),
	\end{align*}
	since $R_{(z,t),(x,t')} \subseteq R_{(x,t-|t'-t|),(x,t')}$ for all
	$|z-x|\leq |t'-t|$. Similarly if $r \leq t' \leq t$,
	\[
	S_n(r,t,f;\omega)(x,y) \leq \sup_{|z-x|\leq |t'-t|} S_n(r,t',f;\omega)(z,y)
	+\frac{1}{n}
	L^{\uparrow}\left(\omega_{\lfloor ny \rfloor} \cap n
	R_{(x,t'-|t'-t|),(x,t)}\right).
	\]
	In any case, since $R_{(x,t-2\delta),(x,t+\delta)}$ contains both
	$R_{(x,t-|t'-t|),(x,t')}$ and $R_{(x,t'-|t'-t|),(x,t)}$,
	\begin{align*}
	&\sup_{\substack{r,t' \in [0,T] \\ r \leq t \wedge
			t',
			\,  |t'-t|
			\leq \delta}}\left| S_n(r,t',f;\omega)(x,y) -
	S_n(r,t,f;\omega)(x,y) \right| \\
	&\leq \sup_{\substack{z \in [x-\delta,x+\delta] \\ r \leq t\wedge t'}}
	\left|
	S_n(r,t\wedge t',f;\omega)(z,y) - S_n(r,t\wedge t',f;\omega)(x,y) \right|
	+ \frac{1}{n} L^{\uparrow}\left(\omega_{\lfloor ny \rfloor} \cap n
	R_{(x,t-2\delta),(x,t+\delta)}\right).
	\end{align*}
	Therefore, by Proposition \ref{prop:difhGW} and Lemma \ref{lem:lis}, since
	$\Leb \left(R_{(x,t-2\delta),(x,t+\delta)}\right) = (3\delta)^2/2$,
	\begin{equation}\label{eq:s,t,t',x,y}
	\begin{aligned}
	\limsup_{n \to \infty}\sup_{\substack{r,t' \in [0,T] \\ r \leq t \wedge
			t',
			\,  |t'-t|
			\leq \delta}}& \left| S_n(r,t',f;\omega)(x,y) -
	S_n(r,t,f;\omega)(x,y) \right|  \\
	& \leq \sup_{\substack{x_1,x_2 \in [x-\delta-T,x+\delta + T] \\
			|x_2-x_1|\leq 2 \delta}}|f(x_1,y)-f(x_2,y)| + C \, \sqrt{\delta} +
			c \,
	\frac{3}{\sqrt{2}} \delta \ .
	\end{aligned}
	\end{equation}
	The right-hand side tends to $0$ when $\delta$ goes to $0$. To finish off
	the proof of \eqref{eq:s,t,t'}, we need to get a uniform control in $(x,y)
	\in [-R,R]^2$. To do this, we cover the rectangle $[-R,R]^2$ by a finite
	union of balls of radius $\delta$. Let $(x_1,y_1),\cdots,(x_p,y_p)$ be the
	centers of these balls. By \eqref{eq:asympteq,x,y}, for any $i$,
	\[
	\begin{aligned}
	\limsup_{n \to \infty}\sup_{\substack{(r,r') \in \T \\ (x,y)\in
			\mathcal{B}((x_i,y_i),\delta)}} &\left| S_n(r,
			r',f;\omega)(x,y) -
	S_n(r, r',f;\omega)(x_i,y_i)
	\right| \\
	&\leq \sup_{\substack{x_1,x_2\in[x-\delta-T,x+\delta+T]
			\\ |x_2-x_1| \leq
			2 \delta}} |f(x_1,y)-f(x_2,y)| +
	C \, \sqrt{\delta} + \delta,
	\end{aligned}
	\]
	This bound proves the uniform control in $(x,y) \in
	\mathcal{B}((x_i,y_i),\delta)$. Since
	\eqref{eq:s,t,t',x,y} holds simultaneously for all $(x_i,y_i)$,
	\eqref{eq:s,t,t'} holds for any $\delta>0$ chosen small enough.

	\medskip

	Let us now prove
	\eqref{eq:s,s',t}. If $s
	\leq s' \leq u$, by Lemma \ref{lem:markovGW},
	\begin{align*}
	&& S_n(s,u,f;\omega) &= \frac{1}{n}h\left(n\cdot,\lfloor n \cdot
	\rfloor,n(u-s'),h(n\cdot,\lfloor n \cdot \rfloor,n(s'-s),\varphi_n^f;
	\tau_{ns} \omega); \tau_{ns'} \omega \right),\\
	\text{and} && S_n(s',u,f;\omega) &= \frac{1}{n} h(n\cdot,\lfloor n \cdot \rfloor,n(u-s'),\varphi_n^f;
	\tau_{ns'} \omega).
	\end{align*}
	Therefore, by Corollary \ref{coro:asymlocal}, there exists $N(\omega) \in \N$ such that for all
	$n \geq N(\omega)$ and all $s \leq s' \leq u \leq T$,
	\begin{align*}
	\| S_n(s,u,f;\omega)&-S_n(s',u,f;\omega)\|_{\infty}^{[-R,R]^2} \\ &\leq
	\left\|
	\frac{1}{n}h(n\cdot,\lfloor n \cdot \rfloor,n(s'-s),\varphi_n^f; \tau_{ns}
	\omega)-\frac{1}{n}\varphi_n^f(n\cdot,\lfloor n \cdot
	\rfloor)\right\|_{\infty}^{[-R-\alpha (u-s'),R + \alpha (u-s')]^2} \\
	& \leq \left\|
	S_n(s,s',f;\omega)-S_n(s,s,f;\omega)\right\|_{\infty}^{[-R-\alpha T,R +
		\alpha T]^2}.
	\end{align*}
	We can do similarly for $s'\leq s \leq u$ and finally get that for all $n
	\geq N(\omega)$,
	\begin{align*}
	\sup_{\substack{u,s' \in [0,T] \\ u \geq s \vee s' ,
			\, |s'-s|
			\leq
			\delta}}& \|
	S_n(s,u,f;\omega)-S_n(s',u,f;\omega)\|_{\infty}^{[-R,R]^2} \\
	&\leq \sup_{\substack{s' \in [0,T] \\ |s-s'| \leq \delta}} \| S_n(s\wedge
	s',s,f;\omega)-S_n(s\wedge s',s',f;\omega)\|_{\infty}^{[-R-\alpha T,R +
		\alpha T]^2}
	\end{align*}
	and the proof is concluded by the first case \eqref{eq:s,t,t'} (with
	$(t,t')=(s,s')$ and $r= s \wedge s'$).
\end{proof}

This shows the asymptotic equi-continuity of $((s,t) \mapsto
S_n(s,t,f;\omega))_{n \in \N}$. Together with Lemma \ref{lem:prc2} and
Proposition \ref{prop:ascoli}, this concludes the proof of Lemma \ref{lem:prc1}.
\end{proof}

The proof of Proposition \ref{prop:compacity} is complete up to showing
Lemma \ref{lem:asymplipschitz}.
\end{proof}

\begin{proof}[Proof of Lemma \ref{lem:asymplipschitz}]
	Let $\omega \in \Omega_0$, $\eps>0$ and $I \in \N$ such that $2^{-I}
	\leq \eps/2$. By
	definition of the
	metric $D_{\infty}$, for any $f,g \in \bar{\Gamma}$,
	\begin{align*}
	D_{\infty}
	\left(S_n(\cdot,\cdot,f;\omega),S_n(\cdot,\cdot,g;\omega)\right) & =
	\sup_{0\leq s \leq t \leq T}
	\sum_{i=1}^{\infty}
	2^{-i} \left(
	\|S_n(s,t,f;\omega)-
	S_n(s,t,g;\omega)\|_{\infty}^{[-i,i]^2}\wedge 1  \right)\\
	&  \hspace{-1cm}\leq \sup_{0\leq s \leq t \leq T}
	\sum_{i=1}^{I}
	2^{-i} \left(
	\|S_n(s,t,f;\omega)-
	S_n(s,t,g;\omega)\|_{\infty}^{[-i,i]^2}\wedge 1  \right) + \eps/2.
	\end{align*}
	Now, by \eqref{eq:asymplipGW}, there exists $N(\omega)\in \N$ such that for
	all $n \geq N(\omega)$ for all $f,g \in \bar{\Gamma}$ and all $0 \leq s
	\leq t \leq T$,
	\begin{align*}
	\|S_n(s,t,f;\omega)-
	S_n(s,t,g;\omega)\|_{\infty}^{[-i,i]^2} &\leq
	\left\|\frac{1}{n}\varphi_n^f(n\cdot,\lfloor n\cdot \rfloor)-
	\frac{1}{n}\varphi_n^g(n\cdot,\lfloor n\cdot
	\rfloor)\right\|_{\infty}^{[-i-\alpha T,i+\alpha T]^2}\\
	& \leq \|f-g\|_{\infty}^{[-i-\alpha T,i+\alpha T]^2} + 4/n. &&
	\text{by \eqref{eq:approxphiGW}}
	\end{align*}
	Therefore, for all $n \geq \max(N(\omega),\eps/8)$ and all $f,g \in
	\bar{\Gamma}$,
	\begin{align*}
	D_{\infty}
	\left(S_n(\cdot,\cdot,f;\omega),S_n(\cdot,\cdot,g;\omega)\right) &\leq
	\sum_{i=1}^{I} 2^{-i} \left(
	\|f-g\|_{\infty}^{[-i-\alpha T,i+\alpha T]^2} \wedge 1  \right) +
	4/n + \eps/2\\
	& \leq 2^{\lceil \alpha T \rceil} \,  \sum_{i=1}^{\infty} 2^{-i-\lceil
		\alpha T
		\rceil} \left(
	\|f-
	g\|_{\infty}^{[-i-\lceil \alpha T
		\rceil,i+\lceil \alpha T
		\rceil]^2}\wedge 1  \right)+ \eps\\
	& \leq 2^{\lceil \alpha T \rceil} \,
	d_{\infty,c}(f,g) + \eps.
	\end{align*}
\end{proof}

\section{Identification of the limit}

\subsection{Properties of the limit points}

In this section, we are going to show that any subsequential limit of
$(S_{n}(.,.;f,\omega))_{n\in \N}$ (as in Definition
\ref{defi:S_n}) satisfies the sufficient conditions of Proposition
\ref{prop:condaxiom}, most of these properties being automatically
satisfied by the analogous microscopic properties stated in Section
\ref{sec:microscopic} or by Proposition \ref{prop:compacity}
concerning continuity.

\begin{prop}\label{prop:proplimiteGW}
	Let $\omega \in \Omega_0$ and $(n_k)_{k \in \N}$ a subsequence such that
	for all $f \in \bar{\Gamma}$, $(S_{n_k}(\cdot,\cdot;f,\omega))_{k \in \N}$
	converges to a certain $S(\cdot,\cdot;f,\omega)$ in $\F_{\T}$, i.e
	\[
	\forall R>0, \qquad \sup_{\substack{(s,t) \in \T \\ |x|,|y| \leq R }}
	\left|
	S_{n_k}(s,t,f;\omega)(x,y)-S(s,t,f;\omega)(x,y) \right| \underset{k \to
	\infty}{\longrightarrow} 0.
	\]
	Any such limit $(f \mapsto S(s,t,f;\omega))_{0\leq s \leq t \leq T}$
 is a family
	of
	continuous functions from $\bar{\Gamma}$ into itself satisfying the first
	four properties listed in	Proposition \ref{prop:condaxiom}.
        Moreover, for any $f
	\in
	\bar{\Gamma}$, $(s,t,x,y) \mapsto
	S(s,t,f;\omega)(x,y)$ is continuous.
\end{prop}

\begin{proof}
	\noindent \underline{- Continuity :} By Proposition \ref{prop:compacity},
	 for all $f \in
	\bar{\Gamma}$, $(s,t) \mapsto S(s,t,f;\omega)$ is continuous from $\T$ into
	$\bar{\Gamma}$ (which is composed of continuous functions) hence $(s,t,x,y)
	\mapsto
	S(s,t,f;,\omega)(x,y)$ is continuous.
	\noindent \underline{- Translation invariance :}
        For any $c \in \R$,
	$s \leq t$
	and $k \in \N$, by Lemma
	\ref{lem:transinv} and by translation invariance property of $
	\varphi_n^f$ stated in Proposition \ref{prop:defphiGW},
	\begin{align*}
	S_{n_k}(s,t;f+ n_k^{-1} \lfloor n_kc \rfloor,\omega) & =
	S_{n_k}(s,t;f,\omega) + \frac{1}{n_k}\lfloor
	n_kc\rfloor .
	\end{align*}
	When $k$ goes to infinity, the right-hand side tends to
	$S(s,t;f,\omega) + c$ in $(\F(\R^2),d_{\infty,c})$ while the
	left-hand side goes to $S(s,t;f+c,\omega)$ by  Lemma
	\ref{lem:asymplipschitz}.

	\noindent \underline{- Monotonicity :} By \eqref{eq:approxphiGW}, if $f
	\leq g$, then for all $k \in \N$, $\varphi_{n_k}^f \leq
	\varphi_{n_k}^g +
	4$ so by Lemmas \ref{lem:monotoneGW} and \ref{lem:transinv},
	\[
	S_{n_k}(s,t,f;\omega) \leq
	S_{n_k}(s,t,g;\omega) + 4/n_k.
	\]
	Monotonicty follows by taking the limit $k\to\infty$.

	\noindent \underline{- Locality :} It is a direct consequence of Corollary
	\ref{coro:asymlocal} and \eqref{eq:approxphiGW}.

	\noindent \underline{- Semi-group :} the fact that $S(t,t,f)=f$, for all $t
	\in [0,T]$ and $f
	\in \bar{\Gamma}$ is an immediate consequence of \eqref{eq:approxphiGW}.
	Now, for any $0
	\leq r \leq s \leq t \leq T$, we have by Lemma \ref{lem:markovGW},
	\begin{align*}
	S_{n_k}(r,t,f;\omega) = \frac{1}{n_k} h(n\cdot, \lfloor n \cdot
	\rfloor,n(t-s),h(\cdot,
	\cdot,n(s-r),\varphi_{n_k}^f;\tau_{nr}\omega);\tau_{ns}\omega),
	\end{align*}
	and since $S(r,s,f;\omega) \in \bar{\Gamma}$, we can apply
	$S_{n_k}(s,t,\cdot; \omega)$
	and write
	\[
	S_{n_k}(s,t,S(r,s,f;\omega);\omega) = \frac{1}{n_k} h(n\cdot, \lfloor n
	\cdot
	\rfloor,n(t-s),\varphi_{n_k}^{S(r,s,f;\omega)};\tau_{ns}\omega).
	\]
	Therefore, by Corollary \ref{coro:asymlocal}, for all $R\geq0$ and $k$
	large
	enough,
	\begin{align*}
	&\|S_{n_k}(r,t,f;\omega)-S_{n_k}(s,t,S(r,s,f;\omega);\omega)
	\|_{\infty}^{[-R,R]^2} \\
	& \hspace{2cm} \leq \left\| \frac{1}{n_k}h(n\cdot, \lfloor n
	\cdot
	\rfloor,n(s-r),\varphi_{n_k}^f;\tau_{nr}\omega) - \frac{1}{n_k}
	\varphi_{n_k}^{S(r,s,f;\omega)}(n\cdot, \lfloor n
	\cdot
	\rfloor)\right\|_{\infty}^{[-R-\alpha T,R + \alpha T]^2}\\
	& \hspace{2cm}  \leq \left\| S_{n_k}(r,s,f;\omega) -
	S(r,s,f;\omega)\right\|_{\infty}^{[-R-\alpha T,R + \alpha T]^2} +
	\frac{2}{n_k}
	&&\text{by \eqref{eq:approxphiGW}}
	\end{align*}
	which tends to zero when $k$ goes to infinity. Consequently, for all $R\geq 0$,
	\begin{align*}
	& \left\| S(r,t,f;\omega) -
	S(s,t,S(r,s,f;\omega);\omega) \right\|_{\infty}^{[-R,R]^2}  \\
	& \hspace{3cm}= \lim_{k \to
	\infty} \left\| S_{n_k}(r,t,f;\omega) -
	S_{n_k}(s,t,S(r,s,f;\omega);\omega) \right\|_{\infty}^{[-R,R]^2} = 0,
	\end{align*}
	which concludes the
	proof of the semi-group property.
\end{proof}

\subsection{Hydrodynamic limit for linear initial profiles}
The only condition missing to apply Proposition \ref{prop:condaxiom} is the
compatibility with linear initial profiles. We start with the following result:
\begin{prop}\label{prop:limitelinearGW}
  For all $\rho \in \R \times (-1,0)$, all $t\in[0,T]$
	and all $(x,y) \in
	\R^2$:
	\begin{equation}\label{eq:convergelimhydrolinear}
	\omega-\mathrm{a.s} \qquad S_{n}(0,t;f_{\rho},\omega)(x,y) \underset{n \to
	\infty}{\longrightarrow}
	f_{\rho}(x,y) + t \, v(\rho),
      \end{equation}
	with $f_{\rho} := (x,y) \mapsto \rho\cdot(x,y)$.
      \end{prop}

Before proving this Proposition, let us show the following Corollary that gives
the compatibility with linear solutions.
\begin{coro}
	\label{coro:complinearGW}
	There exists $\Omega_1 \subseteq \Omega$ of probability one such that for
	all $\omega \in \Omega_0 \cap \Omega_1$, if $(n_k)_{k \in \N}$ is a
	subsequence such that for all $f \in \bar{\Gamma}$,
	$(S_{n_k}(\cdot,\cdot;f,\omega))_{k \in \N}$ converges towards
	$S(\cdot,\cdot;f,\omega)$ in $\F{_\T}$, then
	\begin{equation}
	\forall \rho \in \R \times [-1,0] \quad \forall 0 \leq s \leq t \leq T
	 \qquad S(s,t,f_{\rho};\omega) =
	f_{\rho} +(t-s)v(\rho).
	\end{equation}
\end{coro}

\begin{proof}[Proof of Corollary \ref{coro:complinearGW}]
	By Proposition \ref{prop:limitelinearGW}, there exists a subset $\Omega_1
	\subseteq \Omega$ of probability one such that
	\eqref{eq:convergelimhydrolinear} holds for any	$\rho,t,x,y$ in a countable
	dense subset of their respective set of definition. Therefore,  for all
	$\omega \in \Omega_0 \cap \Omega_1$,
	any subsequential limit $S(\cdot,\cdot,\cdot;\omega)$ of
	$S_n(\cdot,\cdot,\cdot;\omega)$ satisfies
	that
	for any such
	$\rho,t,x,y$,
	\[
	S(0,t,f_{\rho};\omega)(x,y) =
	f_{\rho}(x,y) +t \, v(\rho).
	\]
	By continuity with respect to $(t,x,y)$ of both sides (by Proposition
	\ref{prop:proplimiteGW}), this holds actually for all $(x,y) \in \R^2$ and
	$t \in [0,T]$. Similarly, by continuity of $\rho \mapsto v(\rho)$ (defined
	in \eqref{eq:speedfunction}) and of $\rho \mapsto f_{\rho}$	 on $\R \times
	[-1,0]$
	(including the endpoints of the interval)
	for the
	topology of convergence on all compact sets
	and by continuity of $f \mapsto
	S(0,t,f;\omega)$ for the same topology (still by Proposition
	\ref{prop:proplimiteGW}) we
	deduce that it holds also for all $\rho \in \R \times [-1,0]$. Finally, we
	get the result for any $s>0$ by the semi-group property
	satisfied by $S$ (by Proposition \ref{prop:proplimiteGW}):
	\begin{align*}
	f_{\rho} +t \, v(\rho) &= S(0,t,f_{\rho};\omega) =
	S(s,t,S(0,s,f_{\rho},\omega);\omega) \\
	& = S(s,t,f_{\rho}+s \, v(\rho) ;\omega) = S(s,t,g_{\rho} ;\omega) +s \,
	v(\rho) \qquad \text{by translation
	invariance property}
	\end{align*}
	and thus  $S(s,t,f_{\rho} ;\omega) = f_{\rho} +(t-s) \,
	v(\rho)$.
\end{proof}

\begin{proof}[Proof of Proposition \ref{prop:limitelinearGW}]
  This proof requires the knowledge on equilibrium measures developed
  by Prähofer and Spohn in
  \cite{prahoferthesis,prahofer1997exactly}. As in section
  \ref{sec:equilibrium}, we note $\varphi_{M,N,\rho}$ the height
  function with asymptotic average slope
  $\rho \in \R \times (-1,0)$ (in the thermodynamic limit
    $N\to\infty,M\to\infty$) and whose gradients are stationary w.r.t
  time for the periodised Gates-Westcott dynamic (i.e the Poisson
  point process is periodised on a torus of size $2M$ and $2N$ and
  noted $[\omega]^{M,N}$ as in \eqref{eq:poissonperio}). There are two
  key ingredients in this proof: to show that, in the limit
  $M,N \to \infty$, $\varphi_{M,N,\rho}$ approaches $f_{\rho}$ in the
  sense of \eqref{eq:hypcondinitialGW} and that
  $n^{-1}h(0,0, nt;\varphi_{M,N,\rho},[\omega]^{M,N})$ approaches
  $f_{\rho}+t \, v(\rho)$. From \eqref{eq:moyequi} and
  \eqref{eq:speedequilibrum}, this is true on average. It remains to
  show concentration via variance estimates as in the next Lemmas.

\begin{lem}\label{lem:varh0}
	For any $\rho \in \R \times (-1,0)$ and $t \geq 0$,
	\begin{equation}
	\limsup_{N \to \infty} \limsup_{M \to \infty}  \Var \left(h\left(0,0,
	t;\varphi_{M,N,\rho},[\omega]^{M,N}\right)\right) = \underset{t \to \infty} \bigo
	\left( \log
	t \right).
	\end{equation}
\end{lem}

\begin{lem}\label{lem:uniformequilibrium}
	For any $\rho \in \R \times (-1,0)$, any $\eps>0$, any $n \in \N^*$ and any compact set $K \subseteq \R^2$,
	\begin{equation}
		\limsup_{N \to \infty} \limsup_{M \to \infty} \Pro \left( \sup_{(x,y)
		\in
		K}\left|\frac{1}{n}\varphi_{M,N,\rho}(nx,\lfloor
	ny \rfloor) - f_{\rho}(x,y)\right| \geq \eps \right) = \underset{n \to
	\infty}{\bigo} \left( \frac{\log n}{n^2}\right).
	\end{equation}
\end{lem}

  Let us admit first these Lemmas and finish the proof of Proposition
  \ref{prop:limitelinearGW}. Let us fix $\rho,t,x,y$ as in Proposition
  \ref{prop:limitelinearGW}. We also fix $\eps>0$ and $n \in \N$. For any $M,N
  \in
  \R_+$,
\begin{equation}\label{eq:decompequilibre}
\begin{aligned}
\Pro (|S_{n}(0,t;f_{\rho},\omega)&(x,y) -
f_{\rho}(x,y) -
t v(\rho)| \geq 2 \eps ) \\
& \leq \Pro\left(S_{n}(0,t;f_{\rho},\omega)(x,y)
\neq
S_{n}(0,t;f_{\rho},[\omega]^{M,N})(x,y)  \right) && \rbrace
\mathrm{A} \\
& + \Pro \left(|n^{-1}h(nx,\lfloor ny
\rfloor, nt;\varphi_{M,N,\rho},[\omega]^{M,N}) - f_{\rho}(x,y) -
t v(\rho)| \geq \eps \right) && \rbrace \mathrm{B}  \\
& + \Pro \left(|S_{n}(0,t;f_{\rho},[\omega]^{M,N})(x,y) - n^{-1}h(nx,\lfloor ny
\rfloor, nt;\varphi_{M,N,\rho},[\omega]^{M,N})| \geq \eps \right). && \rbrace \mathrm{C}
\end{aligned}
\end{equation}
Let us bound the limsup when $M,N$ goes to infinity of the three terms of the
r.h.s called A, B and C.

\medskip

\noindent A) The first term is easy to control thanks to the linear propagation
of
	information. For any $M,N$ large enough,
	$[-M,M) \times \llbracket -N , N -1 \rrbracket$ contains $\mathcal{B}((nx,ny),
	\alpha_{M,N} \, nt)$ with $\alpha_{M,N}:=(M\wedge N)/(2nt)$. For such
	$M,N$, if
	$\omega
	\in A_{n(x,y),0,nt,\alpha_{M,N}}$
	(defined in \eqref{eq:defA}), then
	$S_{n}(0,t;f_{\rho},\omega)(x,y) =
	S_{n}(0,t;f_{\rho},[\omega]^{M,N})(x,y)$. Consequently, by
	\eqref{eq:unionlocal} and since $\alpha_{M,N}$ tends to infinity when $M,N$
	tend to infinity,
	\begin{equation}\label{eq:contr1}
	\Pro \left(
	S_{n}(0,t;f_{\rho},\omega)(x,y) \neq
	S_{n}(0,t;f_{\rho},[\omega]^{M,N})(x,y) \right) \leq \Pro \left(
	^c\!A_{n(x,y),0,nt,\alpha_{M,N}} \right) \underset{M,N \to
	\infty}{\longrightarrow}0.
	\end{equation}

	\medskip

\noindent B)
	 Let us write $h(nx,\lfloor ny
	\rfloor, nt;\varphi_{M,N,\rho})$ for $h(nx,\lfloor ny
	\rfloor, nt;\varphi_{M,N,\rho},[\omega]^{M,N})$. By Chebyshev's inequality,
	\begin{align*}
	&\Pro ( |n^{-1}h(nx,\lfloor ny
	\rfloor, nt;\varphi_{M,N,\rho}) - f_{\rho}(x,y) -
	t v(\rho)| \geq \eps ) \\
	& \leq \eps^{-2} \, \Esp \left[
	|n^{-1}h(nx,\lfloor ny
	\rfloor, nt;\varphi_{M,N,\rho}) - f_{\rho}(x,y) -
	t v(\rho)|^2 \right] \\
	& = \eps^{-2} \left( \Esp \left[ n^{-1}h(nx,\lfloor ny
	\rfloor, nt;\varphi_{M,N,\rho}) \right] - f_{\rho}(x,y) -
	t v(\rho) \right)^2  + \eps^{-2} \frac{\Var \left(h(nx,\lfloor ny
	\rfloor, nt;\varphi_{M,N,\rho}) \right)}{n^2}.
	\end{align*}
	By \eqref{eq:speedequilibrum}, the first term of the r.h.s in the last
	equality goes to zero when $M,N$ tends to infinity. To treat the second term, we write $h(nx,\lfloor ny
		\rfloor, nt;\varphi_{M,N,\rho}) = h(0, 0
			, nt;\varphi_{M,N,\rho}) +  h(nx,\lfloor ny
				\rfloor, nt;\varphi_{M,N,\rho})-h(0, 0
				, nt;\varphi_{M,N,\rho})$ and use that the variance of the sum is smaller than twice the
	sum of the variances:
	\begin{align*}
	\Var (h(nx,\lfloor ny &
	\rfloor, nt;\varphi_{M,N,\rho}) ) \\
	&\leq 2 \,\Var
	\left(h(0, 0
	, nt;\varphi_{M,N,\rho}) \right) + 2\,\Var
	\left(h(nx,\lfloor ny
	\rfloor, nt;\varphi_{M,N,\rho})-h(0, 0
	, nt;\varphi_{M,N,\rho}) \right)  \\
	& = 2 \, \Var
	\left(h(0, 0
	, nt;\varphi_{M,N,\rho}) \right) + 2\, \Var
	\left(\varphi_{M,N,\rho}(nx,\lfloor ny
	\rfloor) \right). \qquad \text{by \eqref{eq:stationary}}
	\end{align*}
	The first term of the r.h.s is controlled by Lemma \ref{lem:varh0} and the
	second by \eqref{eq:logfluct}. Therefore,
	\begin{equation}\label{eq:contr2}
	\limsup_{N \to \infty} \limsup_{M \to \infty} \Pro \left(
	\left|n^{-1}h(nx,\lfloor ny
	\rfloor, nt;\varphi_{M,N,\rho}) - f_{\rho}(x,y) -
	t v(\rho)\right| \geq \eps \right) = \underset{n \to \infty}{\bigo}\left( \frac{\log
	n}{n^2} \right).
	\end{equation}
	\medskip

\noindent C)
	 By Lemma \ref{lem:locality} and  by \eqref{eq:approxphiGW}, for any
	$n \geq 2/\eps$
	\begin{align*}
	 \Pro &\left( |S_{n}(0,t;f_{\rho},[\omega]^{M,N})(x,y) - n^{-1}h(nx,\lfloor
	ny
	\rfloor, nt;\varphi_{M,N,\rho},[\omega]^{M,N})| \geq \eps \right) \\
	& \leq \Pro \left( \sup_{(x',y') \in \mathcal{B}((x,y),\alpha
	t)}\left|f_{\rho}(x',y') -
	\frac{1}{n}\varphi_{M,N,\rho}(nx',\lfloor
	ny' \rfloor)\right| \geq \eps/2 \right) + \underset{n \to
	\infty}{\bigo}\left(
	e^{-\gamma
	\, n} \right).
	\end{align*}
	Consequently, by Lemma \ref{lem:uniformequilibrium},
	\begin{equation}\label{eq:contr3}
	\limsup_{N \to \infty} \limsup_{M \to \infty} \Pro \left(
	|S_{n}(0,t;f_{\rho},[\omega]^{M,N})(x,y) - n^{-1}h(nx,\lfloor
	ny
	\rfloor, nt;\varphi_{M,N,\rho},[\omega]^{M,N})| \geq \eps \right) =
	\underset{n \to \infty}{\bigo}\left( \frac{\log
		n}{n^2} \right).
	\end{equation}
Altogether, by taking the limsup when $M,N$ goes to infinity in
\eqref{eq:decompequilibre} and by \eqref{eq:contr1}, \eqref{eq:contr2} and
\eqref{eq:contr3},
\[
\Pro \left( |S_{n}(0,t;f_{\rho},\omega)(x,y) -
f_{\rho}(x,y) -
t v(\rho)| \geq 2\eps \right) = \underset{n \to \infty}{\bigo}\left( \frac{\log
	n}{n^2} \right),
\]
and the proof of Proposition \ref{prop:limitelinearGW} follows from
Borel-Cantelli Lemma.
\end{proof}

Now, as promised, we prove Lemmas \ref{lem:varh0} and
\ref{lem:uniformequilibrium}.

\begin{proof}[Proof of Lemma \ref{lem:varh0}]
	Again we write $h(x,y,t;\varphi_{M,N,\rho})$ instead of
	$h(x,y,t;\varphi_{M,N,\rho},[\omega]^{M,N})$. For any rectangle
	$\Lambda_R= [-R,R] \times \llbracket -R , R
	\rrbracket$ with $R>0$ and any $t
	\geq 0$ if we define
	\[
	h(\Lambda_R,t;\varphi_{M,N,\rho}) := \sum_{y=-R}^R \int_{-R}^R
	h(x,y,t;\varphi_{M,N,\rho}) \d{x},
	\]
	then it is easy to see
	\begin{equation}
	h(\Lambda_R,t;\varphi_{M,N,\rho}) - h(\Lambda_R,0;\varphi_{M,N,\rho})  =
	\int_0^t
	\left(N^+_{M,N,\rho}(\Lambda_R,s)+ N^-_{M,N,\rho}(\Lambda_R,s)\right)
	\d{s},
	\end{equation}
	where $N^{\pm}_{M,N,\rho}(\Lambda_R,s)$ is the number of antikinks/kinks in the
	domain $\Lambda_R$ at time $s$ for the dynamic starting from
	$\varphi_{M,N,\rho}$. Then,
	\begin{align*}
	\Var &\left( h(\Lambda_R,t;\varphi_{M,N,\rho})  -
	h(\Lambda_R,0;\varphi_{M,N,\rho}) \right) \\
	&  = \int_0^t \int_0^t
	\Cov \left((N^{+}+N^{-})_{M,N,\rho}(\Lambda_R,s),
	(N^{+}+N^{-})_{M,N,\rho}(\Lambda_R,s') \right) \d{s} \d{s'} \\
	& \leq t^2 \, \Var \left( (N^{+}+N^{-})_{M,N,\rho}(\Lambda_R,0)
	\right) \leq 2 t^2 \, \left(\Var \left( N^{+}_{M,N,\rho}(\Lambda_R,0)
	\right) + \Var \left( N^{-}_{M,N,\rho}(\Lambda_R,0)
	\right) \right),
	\end{align*}
	where the two last inequalities hold by Cauchy-Schwarz inequality and by
	stationarity with respect to time.
	Therefore,  by \eqref{eq:varkinks} applied for $R=t$,
	\begin{equation}\label{eq:varhlambda}
	\limsup_{N \to \infty} \limsup_{M \to \infty} \Var \left(
	h(\Lambda_t,t;\varphi_{M,N,\rho})  -
	h(\Lambda_t,0;\varphi_{M,N,\rho}) \right) =\underset{R \to
		\infty}{\bigo} \left(t^4 \log t\right).
	\end{equation}
	Now we are going to compare $h(\Lambda_t,t;\varphi_{M,N,\rho})  -
	h(\Lambda_t,0;\varphi_{M,N,\rho})$ with
	$t^2 \,
	h(0,0,t;\varphi_{M,N,\rho})$, using the    logarithmic bound
	\eqref{eq:logfluct} on fluctuations. We can write
	\begin{equation}\label{eq:comparhlambda}
	\begin{aligned}
	2t(2\lfloor t \rfloor +1)
	h(0,0,t;\varphi_{M,N,\rho})	& = \sum_{y=- \lfloor t \rfloor}^{\lfloor t
	\rfloor} \int_{-t}^t
	(h(0,0,t;\varphi_{M,N,\rho}) - h(x,y,t;\varphi_{M,N,\rho})) \d{x} \\
	& \hspace{0.5cm} + \sum_{y=- \lfloor t \rfloor}^{\lfloor t \rfloor}
	\int_{-t}^t
	\varphi_{M,N,\rho}(x,y)\d{x} +
	h(\Lambda_t,t;\varphi_{M,N,\rho})-h(\Lambda_t,0;\varphi_{M,N,\rho}).
	\end{aligned}
	\end{equation}
	By Cauchy-Schwarz inequality and by stationarity \eqref{eq:stationary}, for any
	$(x,y),(x',y') \in \Lambda_t$,
	\begin{align*}
	\limsup_{N \to \infty} \limsup_{M \to \infty}
	\Cov&(h(0,0,t;\varphi_{M,N,\rho}) -
	h(x,y,t;\varphi_{M,N,\rho}),h(0,0,t;\varphi_{M,N,\rho}) -
	h(x',y',t;\varphi_{M,N,\rho}))\\
	& \leq \limsup_{N \to \infty} \limsup_{M \to \infty} \sqrt{\Var \left(
	\varphi_{M,N,\rho}(x,y)\right)} \, \sqrt{\Var \left(
	\varphi_{M,N,\rho}(x',y')\right)}  \\
	&= \underset{t \to \infty}{\bigo}(\log t), \qquad \text{by \eqref{eq:logfluct}}
	\end{align*}
	and thus
	\[
	\limsup_{N \to \infty} \limsup_{M \to \infty} \Var \left( \sum_{y=- \lfloor
	t
		\rfloor}^{\lfloor t \rfloor} \int_{-t}^t
	(h(0,0,t;\varphi_{M,N,\rho}) - h(x,y,t;\varphi_{M,N,\rho})) \d{x} \right) =
	\underset{t \to \infty}{\bigo}(t^4\log t).
	\]
	By the same argument, we get
		\[
	\limsup_{N \to \infty} \limsup_{M \to \infty} \Var \left( \sum_{y=- \lfloor
		t
		\rfloor}^{\lfloor t \rfloor} \int_{-t}^t
	\varphi_{M,N,\rho}(x,y) \d{x} \right) =
	\underset{t \to \infty}{\bigo}(t^4\log t).
	\]
	Therefore, using \eqref{eq:comparhlambda}, \eqref{eq:varhlambda} and that
 	the variance of the sum of three terms is less than three times the sum of
 	the variances,
	\begin{align*}
	(2t(2\lfloor t \rfloor +1))^2 \limsup_{N\to \infty} \limsup_{M \to \infty}
	\Var
	\left(
	h(0,0,t;\varphi_{M,N,\rho}) \right) = \underset{t \to
	\infty}{\bigo}(t^4\log t),
	\end{align*}
	which concludes the proof of the Lemma.
\end{proof}

\begin{proof}[Proof of Lemma \ref{lem:uniformequilibrium}]
	Since $K$ is compact, for any $\delta>0$, we can cover $K$ by a finite
	number $l_{\delta}\in \N$ of balls $\mathcal{B}((x_i,y_i),\delta)_{1 \leq i
	\leq l_{\delta}}$. Fix
	$i \in \llbracket 1 , l_{\delta}
	\rrbracket$ and $(x,y) \in \mathcal{B}((x_i,y_i),\delta)$. For all $Y \in
	\llbracket\lfloor n(y_i-\delta) \rfloor, \lfloor n(y_i+\delta) \rfloor
	\rrbracket$,
	\begin{align*}
	&\left|f_{\rho}(x,y) -
	\frac{1}{n}\varphi_{M,N,\rho}(nx,\lfloor
	ny \rfloor)\right| \leq \left|f_{\rho}(x,y)
	-f_{\rho}(x_i,y_i)\right|
	+ \left|f_{\rho}(x_i,y_i) -
	\frac{1}{n}\varphi_{M,N,\rho}(nx_i,\lfloor
	ny_i \rfloor)\right| \\
	& \hspace{1.5cm} + \left|\frac{1}{n}\varphi_{M,N,\rho}(nx_i,\lfloor
	ny_i \rfloor)
	-\frac{1}{n}\varphi_{M,N,\rho}(nx_i,Y)\right| +
	\left|\frac{1}{n}\varphi_{M,N,\rho}(nx_i,Y)
	-\frac{1}{n}\varphi_{M,N,\rho}(nx,Y)\right| \\
	& \hspace{1.5cm}+
	\left|\frac{1}{n}\varphi_{M,N,\rho}(nx,Y)
	-\frac{1}{n}\varphi_{M,N,\rho}(nx,\lfloor ny \rfloor)\right|\\
	& \leq (|\rho_1|+|\rho_2|+3) \, \delta + \left|f_{\rho}(x_i,y_i) -
	\frac{1}{n}\varphi_{M,N,\rho}(nx_i,\lfloor
	ny_i \rfloor)\right| + \frac{1}{n}
	\left|(N^+_{M,N,\rho}-N^-_{M,N,\rho})(nI_{x,x_i}\times\{Y\})\right|\\
	& \leq (|\rho_1|+|\rho_2|+3) \, \delta + \left|f_{\rho}(x_i,y_i) -
	\frac{1}{n}\varphi_{M,N,\rho}(nx_i,\lfloor
	ny_i \rfloor)\right| + \frac{1}{n}
	(N^+_{M,N,\rho}+N^-_{M,N,\rho})(n[x_i-\delta,x_i+\delta]\times\{Y\}),
	\end{align*}
	where $\rho = (\rho_1,\rho_2)$, $N^{\pm}_{M,N,\rho}(D)$ is the number of antikinks/kinks of
	$\varphi_{M,N,\rho}$ in
	a domain $D$ and $I_{x,x_i} = [x\wedge x_i, x \vee x_i]$ (the second
	inequality  holds
         because the height slope in the  $y$ direction is bounded by $1$). One
         could
	simply choose
	$Y=\lfloor n y_i \rfloor$ in the last inequality and try to control the
	variance of
	$(N^+_{M,N,\rho}+N^-_{M,N,\rho})(n[x_i-\delta,x_i+\delta]\times\{\lfloor n
	y_i \rfloor\})$ for large $n$ (after sending $M,N$ to infinity) but it
	is not obvious to get a bound better than $\bigo(n)$ (which is
	insufficient).
	Instead, we	average the last
	inequality for all possible values of $Y$ in $
	\llbracket\lfloor n(y_i-\delta) \rfloor, \lfloor n(y_i+\delta) \rfloor
	\rrbracket$ in order to get
		\begin{equation}\label{eq:contrunif1}
	\begin{aligned}
	\sup_{(x,y) \in \mathcal{B}((x_i,y_i),\delta)} \left|f_{\rho}(x,y) -
	\frac{1}{n}\varphi_{M,N,\rho}(nx,\lfloor
	ny \rfloor)\right|& \leq (|\rho_1| + |\rho_2|+3) \, \delta +
	\left|f_{\rho}(x_i,y_i) -
	\frac{1}{n}\varphi_{M,N,\rho}(nx_i,\lfloor
	ny_i \rfloor)\right|\\
& 	+ \frac{1}{(2\delta n -1)n}
	\left(N^+_{M,N,\rho} + N^-_{M,N,\rho}\right) ((nx_i,\lfloor ny_i
	\rfloor)+\Lambda_{n\delta+1}),
	\end{aligned}
	\end{equation}
	where $\Lambda_{n\delta+1}$ is the rectangle defined as at the beginning of the proof of Lemma \ref{lem:varh0}. Now, we know
	from \eqref{eq:moyequi}, from \eqref{eq:logfluct} and from
	Bienaymé–Chebyshev inequality that
	\begin{equation}\label{eq:contrunif2}
	\limsup_{N \to \infty} \limsup_{M \to \infty} \Pro \left(
	\left|f_{\rho}(x_i,y_i)
	-
	\frac{1}{n}\varphi_{M,N,\rho}(nx_i,\lfloor
	ny_i \rfloor)\right| \geq \eps/4 \right) = \underset{n \to
	\infty}{\bigo}\left( \frac{\log
		n}{n^2} \right).
	\end{equation}
	Moreover, by \eqref{eq:varkinks} and by invariance by
        translation of the stationary measures,
	\[
	\lim_{N \to \infty} \lim_{M \to \infty} \Var \left(
	(N_{M,N,\rho}^++N_{M,N,\rho}^-)((nx_i,\lfloor ny_i\rfloor
	)+\Lambda_{n\delta+1}) \right) =
	\underset{n \to
		\infty}{\bigo}(n^2\log n).
            \]
            Besides, since the sum of the asymptotic kink and antikink
            densities is equal to the average speed~$v(\rho)$,
	\[
	\lim_{N \to \infty} \lim_{M \to \infty} \Esp \left[
	\left( N_{M,N,\rho}^++N_{M,N,\rho}^-\right)((nx_i,\lfloor ny_i\rfloor
	)+\Lambda_{n\delta+1}) \right] \underset{n \to \infty}{\sim} (2n\delta)^2
	\, v(\rho).
	\]
	Note that the two previous limits exist as explained in Appendix \ref{sec:kak_cor_eq}.
	Dividing by $(2\delta n -1)n$ and using Bienaymé–Chebyshev inequality yields
	\begin{equation}\label{eq:contrunif3}
	\limsup_{N \to \infty} \limsup_{M \to \infty} \Pro \left(
	\frac{1}{(2\delta n
	-1)n}(N_{M,N,\rho}^{+}+N_{M,N,\rho}^-)\left((nx_i,\lfloor ny_i\rfloor
	)+\Lambda_{n\delta+1} \right) \geq 2\delta v(\rho) + \eps/4 \right) =
	\underset{n
	\to
		\infty}{\bigo}\left( \frac{\log
		n}{n^2} \right).
	\end{equation}
	From \eqref{eq:contrunif1}, \eqref{eq:contrunif2} and
	\eqref{eq:contrunif3},  we get that for any $\delta>0$,
	\begin{align*}
	& \limsup_{N \to \infty} \limsup_{M \to \infty} \Pro \left( \sup_{(x,y) \in
K} \left|f_{\rho}(x,y) -
	\frac{1}{n}\varphi_{M,N,\rho}(nx,\lfloor
	ny \rfloor)\right| \geq C_{\rho} \delta +  \eps/2 \right)\\
	& \leq \sum_{i=1}^{l_{\delta}} \limsup_{N \to \infty} \limsup_{M \to
	\infty} \Pro
	\left( \sup_{(x,y) \in
		\mathcal{B}((x_i,y_i),\delta)} \left|f_{\rho}(x,y) -
	\frac{1}{n}\varphi_{M,N,\rho}(nx,\lfloor
	ny \rfloor)\right| \geq C_{\rho} \delta +  \eps/2 \right)  = \underset{n
	\to
		\infty}{\bigo}\left( \frac{\log
		n}{n^2} \right),
	\end{align*}
	with $C_{\rho} := |\rho_1| + |\rho_2| + 3 + 2 v(\rho)$ which concludes the
	proof by setting $\delta = \eps/(2C_{\rho})$.
      \end{proof}

\subsection{Conclusion of the proof of Theorem \ref{theo:principalGW}}
Propositions \ref{prop:compacity} (compactness) and Proposition
\ref{prop:proplimiteGW} together with Corollary
\ref{coro:complinearGW} provide all necessary ingredients to conclude the
proof of Theorem \ref{theo:principalGW}.

\begin{prop}\label{prop:conclGW}
	For all $\omega \in \Omega_0 \cap \Omega_1$, all $f \in \bar{\Gamma}$ and
	all $R,T>0$,
	\begin{equation}\label{eq:convergenceS_n}
	\sup_{|x|,|y| \leq R, t \in [0,T]}|S_{n}(0,t;f,\omega)(x,y)-
	u(x,y,t)| \underset{n \to \infty}{\longrightarrow} 0,
	\end{equation}
	where $u$ is the unique viscosity solution of \eqref{eq:hamilton-jacobiGW}.
\end{prop}

\begin{proof}
	Assume that convergence \eqref{eq:convergenceS_n}
	does not hold for some $\omega \in \Omega_0 \cap \Omega_1$, $f
	\in \bar{\Gamma}$ and $R,T>0$. Then, there exists $\eps>0$ and a subsequence
	$(n_k)_{k \in \N}$
	such that
	\begin{equation}\label{eq:contradiconvergence}
	\sup_{|x|,|y| \leq R, t \in
		[0,T]}|S_{n_{k}}(0,t;f,\omega)(x,y)-
	u(x,y,t)| \geq \eps.
	\end{equation}
	By Proposition \ref{prop:compacity}, we can extract
	another subsequence $(n_{k_l})_{l \in \N}$ such that for all $g \in
	\bar{\Gamma}$, the sequence $(S_{n_{k_l}}(\cdot,\cdot;g,\omega))_{l \in \N}$ converges
	towards a certain $S(\cdot,\cdot;g,\omega)$ in $\F{_\T}$. By Proposition
	\ref{prop:proplimiteGW} and Corollary
	\ref{coro:complinearGW}, $(S(s,t,\cdot;\omega))_{0\leq s \leq t \leq T}$
	satisfies all sufficient conditions of Proposition \ref{prop:condaxiom}.
	Therefore, $(x,y,t)\mapsto S(0,t,f;\omega)(x,y) = u(x,y,t)$ is the unique
	viscosity solution of \eqref{eq:hypcondinitialGW} and thus $(x,y,t) \mapsto
	S_{n_{k_l}}(0,t;g,\omega)(x,y)$ converges on all compact sets of $\R^2
	\times [0,T]$ towards $u$ when $l$ goes to infinity which is a
	contradiction with
	\eqref{eq:contradiconvergence}.
\end{proof}

The full proof of
Theorem \ref{theo:principalGW} follows from Proposition \ref{prop:conclGW} and
the fact that, by locality (Corollary~\ref{coro:asymlocal}),
\begin{equation}
\sup_{|x|,|y| \leq R, t \in
		[0,T]} \left|S_{n}(0,t;f,\omega)(x,y)-
	\frac{1}{n} h(n\cdot,\lfloor n\cdot \rfloor,nt,\varphi_n;\omega)\right|  \underset{n \to \infty}{\longrightarrow} 0,
\end{equation}
since both rescaled
initial
height functions
$n^{-1}\varphi_n^f(n\cdot, \lfloor n \cdot \rfloor)$ and
$n^{-1}\varphi_n(n\cdot, \lfloor n \cdot \rfloor)$ converges to
$f$ uniformly on $[-R-\alpha T,R+ \alpha
T]^2$ by
\eqref{eq:approxphiGW} and assumption \eqref{eq:hypcondinitialGW}.

\appendix

\section{Sufficient conditions for viscosity solutions of Hamilton-Jacobi equations}\label{sec:sufcond}
In this section, we give a self-contained proof of Proposition
\ref{prop:condaxiom} which is inspired from \cite[Lemma
5.3]{rezakhanlou2001continuum}
and \cite[Proposition
7.1]{zhang2018domino}.
\begin{proof}
Let us show that $u : (x,t) \mapsto S(0,t,f)(x)$ defined from $\R^2 \times [0,T]$ to $\R$ is a viscosity solution of \eqref{eq:hamilton-jacobiGW}. First of all, by assumption, $u$ in continuous on $\R^d \times [0,T]$. Then, by the Semi-group property:
\[
u(\cdot,0) = S(0,0,g) = g.
\]
We are left to show that $u$ is a subsolution (the proof that $u$ is a
supersolution being identical). Let $\phi \in \mathcal{C}^{\infty}(\R^d
\times(0,T))$ and $(x_0,t_0) \in \R^d \times (0,T)$ such that $\phi(x_0,t_0) =
u(x_0,t_0)$ and $\phi \geq u$ on a neighbourhood of $(x_0,t_0)$.
At first, we introduce the following affine approximation of~$\phi$
around~$x_0$:
\[
\psi(x,t) := \phi(x_0,t) + \nabla \phi(x_0,t_0).(x-x_0).
\]
As $\psi$ and $\phi$ have the same value and derivatives at $(x_0,t_0)$, it is
enough to show that
\begin{equation}\label{eq:eqpsi}
\partial_t \psi(x_0,t_0) \leq v(\nabla \psi(x_0,t_0)),
\end{equation}
by studying $\psi(x_0,t_0) - \psi(x_0,t_0-\delta)$ for small positive $\delta$.

On the one hand, by the semi-group property and the definition of $u$,
\begin{equation}\label{eq:semigrouppsy}
\psi(x_0,t_0) = u(x_0,t_0) = S(t_0-\delta,t_0,u(\cdot,t_0-\delta))(x_0).
\end{equation}

On the other hand, it is easy to show that $\nabla \phi(x_0,t_0)\in \R \times
[-1,0]$ thanks to the assumptions~$\phi \geq u$ around $(x_0,t_0)$ with
equality at $(x_0,t_0)$ and the slopes constraints satisfied by functions in
$\bar{\Gamma}$
such as $u(\cdot,t_0)$. Therefore, by compatibility
with linear
solutions and translation
invariance,
\begin{equation}\label{eq:Spsi}
S(t_0-\delta,t_0,\psi(.,t_0-\delta))(x_0) = \psi(x_0,t_0-\delta) + \delta \,
v(\nabla \psi(x_0,t_0)).
\end{equation}

We are left to compare $S(t_0-\delta,t_0,\psi(.,t_0-\delta))(x_0)$ with
$S(t_0-\delta,t_0,u(\cdot,t_0-\delta))(x_0)$. Thanks to locality and monotony,
this can be done by comparing $\psi(.,t_0-\delta)$ with $u(\cdot,t_0-\delta)$
in
the ball~$\mathcal{B}(x_0,\alpha \, \delta)$. By Taylor expansion of $\phi$ and
$\psi$ at order $2$ around $(x_0,t_0)$,
\begin{align*}
\phi(x,t) & = \psi(x,t) + \mathrm{O}\left( \|x-x_0\|_{\infty}^2 + |t-t_0|^2
\right).
\end{align*}
Moreover, $u \leq \phi$ on a neighbourhood of $(x_0,t_0)$ hence
$u(\cdot,t_0-\delta) \leq \phi(\cdot,t_0-\delta)$ on $\mathcal{B}(x_0,\alpha \,
\delta)$ for $\delta$ small enough. Altogether, there exists $C>0$ such that
for
all $\delta$ small enough,
\begin{equation}\label{eq:condaxiompsi}
\forall x \in \mathcal{B}(x_0,\alpha \, \delta) \qquad u(x,t_0-\delta) \leq
\psi(x,t_0-\delta) + C \, \delta^2.
\end{equation}

Now, we set~$g :=
u(.,t_0-\delta)
\wedge \psi(.,t_0-\delta)$. By locality property (applied at $x_0$ with $R =0$),
\begin{equation}\label{eq:compareS}
 |S(t_0-\delta,t_0,u(.,t_0-\delta))(x_0) - S(t_0-\delta,t_0,g)(x_0)| \leq
 \sup_{x \in \mathcal{B}(x_0,\alpha \, \delta)} |u(x,t_0-\delta) - g(x)| \leq C
 \delta^2,
\end{equation}
where the last inequality holds because of \eqref{eq:condaxiompsi}. Since
$\psi \geq g$,
\begin{flalign*}
S(t_0-\delta,t_0,\psi(.,t_0-\delta))(x_0) & \geq S(t_0-\delta,t_0,g)(x_0) &&
\text{by monotonicity} \\
& \geq S(t_0-\delta,t_0,u(.,t_0-\delta))(x_0) - C \delta^2 && \text{by
\eqref{eq:compareS}} \\
&= \psi(x_0,t_0) - C \delta^2. && \text{by \eqref{eq:semigrouppsy}}
\end{flalign*}
Using \eqref{eq:Spsi}, we finally get
\begin{flalign*}
&& \psi(x_0,t_0-\delta) + \delta v(\nabla \psi(x_0,t_0)) \geq \psi(x_0,t_0) - C
\delta^2 && \ .&&
\end{flalign*}
and then
\[
\partial_t \psi(x_0,t_0) = \lim_{\delta \to 0} \frac{\psi(x_0,t_0) -
\psi(x_0,t_0-\delta)}{\delta} \leq  v(\nabla \psi(x_0,t_0)).
\]

\end{proof}

\section{Stationary kink/antikink correlations and proof of equation \eqref{eq:varkinks}}\label{sec:kak_cor_eq}

In this section, we give more details about the determinantal structure of the stationary measures introduced in Section \ref{sec:equilibrium} and show that the kink/antikink correlations are bounded by the inverse of the distance squared in order to deduce \eqref{eq:varkinks}.

Let us first fix $M$ and $N$, the sizes of the torus, and a slope $\rho = (\rho_1,\rho_2) \in \R \times (-1,0)$. The existence of a stationary height profile $\varphi_{M,N,\rho}$ (with value fixed e.g to $0$ at the origin) whose average slope approaches $\rho$ was already discussed in Section \ref{sec:equilibrium}. The height function (and in particular the kinks and antikinks)  are totally determined by the occupation variables $\eta(x,y)$ for $(x,y)\in \R \times \Z$ that take value $1$ if there is a level line of the height function passing by $(x,y)$ (i.e if $\varphi_{M,N,\rho}(x,y+1) - \varphi_{M,N,\rho}(x,y) = -1$) and $0$ otherwise. In \cite{prahoferthesis}, the author showed, that any moments of the occupation variables can be computed thanks to a determinant: for any $(x_1,y_1),\cdots,(x_m,y_m) \in \R \times \Z$,
\begin{equation}
	\Esp \left[ \eta(x_1,y_1) \cdots \eta(x_m,y_m)  \right] = \det \left(S_{M,N,\rho}(x_k,y_k;x_l,y_l)\right)_{1\leq k,l \leq m},
\end{equation}
where $S_{M,N,\rho}$ is an explicit kernel that somehow simplifies in the infinite volume limit:
\begin{equation}
	\lim_{N \to \infty} \lim_{M \to \infty} S_{M,N,\rho}(x',y';x,y) =
	\left\{
	\begin{aligned}
		& \frac{1}{2\pi} \int_{-\pi \rho_2}^{\pi \rho_2} e^{(x'-x) \eps(k)} e^{i(y'-y)k} \d k && & \text{for} && & x' \geq x \\
		& - \frac{1}{2\pi} \int_{\pi \rho_2}^{2\pi - \pi \rho_2}  e^{(x'-x) \eps(k)} e^{i(y'-y)k} \d k && & \text{for} && & x' < x,
	\end{aligned}
	\right.
\end{equation}
with $\eps(k) = - \eta_s \cos(k) + i \eta_a \sin(k)$ and where $\eta_s >0$, $\eta_a \in \R$ are parameters uniquely determined by $\rho$. In particular, the law of $\varphi_{M,N,\rho}$ admits
an infinite volume limit in the sense that the average of any local function
has a limit as $N\to\infty$ after $M\to\infty$.

Thanks to this determinental structure, Prähofer and Spohn computed the infinite volume limit of the densities of kinks and antikinks and deduced the speed of growth $v(\rho)$ (defined in \eqref{eq:speedfunction}) depending on the slope $\rho$. Furthemore, they computed the covariance (or "structure function") between kinks, antikinks and occupation variables (see  \cite[Equation (6.30)]{prahoferthesis} and \cite[Equation (27) and
(29)]{prahofer1997exactly}). For our purposes, we only need the antikink/antikink and kink/kink covariances between the origin and $(x,y)$ denoted respectively by $S^+_{\rho}(x,y)$ and $S^-_{\rho}(x,y)$ and which can be written as:
\begin{equation}\label{eq:structure_function}
	S^{\pm}_{\rho}(x,y) =  \frac{\eta_{\pm}^2}{(2\pi)^2} \underbrace{\int_{-\rho_2 \pi}^{\rho_2 \pi} e^{|x| \eps(k)}e^{i(\frac{x}{|x|}y \pm 1)k} \d k}_A \times \underbrace{\int_{\rho_2 \pi}^{2\pi -\rho_2 \pi} e^{-|x|\eps(k')}e^{i(\frac{x}{|x|}y \pm 1)k'} \d k'}_B,
\end{equation}
where $\eta_{\pm}$ are positive constants determined by $\rho$. Let us show that
\begin{equation}\label{eq:bound_struct_func}
	S^{\pm}_{\rho}(x,y) = \underset{\|(x,y)\| \to \infty}{\mathrm{O}}\left(\frac{1}{\|(x,y)\|^2}\right).
\end{equation}
Without loss of generality, let us treat the case of $S^+$ and $x \geq 0$. First of all, the modulus of $A$ in \eqref{eq:structure_function} is bounded by $2 \int_{0}^{\rho_2 \pi} e^{-\eta_s x \cos(k) } \d k$ whose asymptotic behavior for large $x$ only depends on the behavior of the integrand around $\rho_2 \pi$ where it attains its maximum. Therefore, by a Taylor approximation, we get that for all $(x,y)$,
\begin{equation}\label{eq:bound_mod_struct}
	|A| \leq 2\int_0^{+\infty} e^{-(\eta_s\cos(\rho_2 \pi)-k/C)x} \d k = 2C \, \frac{e^{-\eta_s\cos( \rho_2 \pi) x}}{x} ,
\end{equation}
for some  constant $C>0$. Now, by integration by parts, we get that
\begin{align*}
A  &= \frac{1}{i(y+1)} \left(\left[ e^{x \eps(k)}e^{i(y + 1)k}\right]_{-\rho_2 \pi}^{\rho_2 \pi} - x \int_{-\rho \pi}^{\rho \pi} (\eta_s \sin(k)+i\eta_a \cos(k)) e^{x \eps(k)}e^{i(y + 1)k} \d k \right)
\end{align*}
and thus, by using \eqref{eq:bound_mod_struct} to bound the second term, we obtain
\begin{equation}
|A|  \leq \frac{1}{|y+1|} \left( 2 e^{-\eta_s\cos( \rho_2 \pi) x} + x \,  \sqrt{\eta_s^2 + \eta_a^2} \, 2C\, \frac{e^{-\eta_s\cos( \rho_2 \pi) x}}{x}\right) \leq C' \frac{e^{-\eta_s\cos( \rho_2 \pi) x}}{|y|},
\end{equation}
for some constant $C'>0$. In any case, we have  that
\begin{equation}
|A| = \underset{\|(x,y)\|\to \infty}{\bigo} \left( \frac{e^{-\eta_s\cos( \rho_2 \pi) |x|}}{\max(|x|,|y|)} \right),
\end{equation}
and similar computations show that
\begin{equation}
|B| = \underset{\|(x,y)\|\to \infty}{\bigo} \left( \frac{e^{\eta_s\cos( \rho_2 \pi) |x|}}{\max(|x|,|y|)} \right),
\end{equation}
which concludes the proof of \eqref{eq:bound_struct_func}, by equivalence of norms on $\R^2$.

Now, let us show how we can deduce \eqref{eq:varkinks}. The variance of the number of antikinks/kinks in the domain $\Lambda_R$ is given by:
\[
\lim_{N \to \infty} \lim_{M \to \infty} \Var( N^{\pm}_{M,N,\rho}(\Lambda_R) ) = \int_{[-R,R]^2} \sum_{y,y' \in \llbracket -R, R \rrbracket } S^{\pm}_{\rho}(x'-x,y'-y) \d x \d x'.
\]
By standard approximation of sums by integrals arguments and by \eqref{eq:bound_struct_func}, the proof of \eqref{eq:varkinks} is concluded thanks to the following inequality:
\[
\int_{[-R,R]^4} \frac{C_1}{\|(x'-x,y'-y)\|^2} \vee M \d x  \d y \d x' \d y' \leq \int_{[-R,R]^2} C_2 \,\log R \d x \d y \leq C_2 \, R^2 \log R,
\]
where $M$ is the sup norm of $S^+_{\rho}$ and $C_1,C_2>0$ are constants chosen large enough.

\section{Longest light-chain of Poisson points}\label{sec:lis}

In this section we give a control on the maximal length of Poisson points in a domain that can be collected by a light-path (as in Definition \ref{defi:light}). Let $\omega \in \Omega$, $k \in \N$, $\ubar{y}=(y_1,\cdots,y_k) \in \Z^k$ and $D$ a bounded domain of $\R^2$. We define the event
\begin{equation}
C^{\uparrow}_{\omega,\ubar{y}}(D) := \left\{ \omega \in \Omega, \ \exists (x_i,t_i)_{1 \leq i \leq k} \in \prod_{i=1}^k \left( \omega_{y_i} \cap D \right), \ \forall i \in \llbracket 1, k-1 \rrbracket \ |x_{i+1}-x_i| \leq t_{i+1}-t_i \right\},
\end{equation}
which means that there exists a light-path that collects at least one point per
set $\omega_{y_i} \cap D$ in a precise order (from $i=1$ to $i=k$). The link
with $L^{\uparrow}$ of Definition \ref{defi:light}  is the following. If
$\ubar{y} = (y,\cdots,y)$ where $y\in
\Z$ appears $k$ times, then
\[
\left\{L^{\uparrow} \left( \omega_y \cap D \right) \geq k \right\} = C^{\uparrow}_{\omega,\ubar{y}}(D).
\]
The next Lemma gives a control on the probability of this event when $D$ is a
light-rectangle (see Definition \ref{defi:light}).

\begin{lem}\label{lem:lisr}
	For any light-rectangle $R \subseteq \R^2$, any $k \in \N$ and any
	$\ubar{y}=(y_1,\cdots,y_k) \in \Z^k$,
	\[
	\Pro\left(C^{\uparrow}_{\omega,\ubar{y}}(R) \right) \leq
	\left(\frac{2e^2\,  \Leb(R)}{k^2}\right)^k.
	\]
\end{lem}
\begin{proof}
	This probability is invariant by translation of $R$ and up to a rotation of angle
	$-\pi/4$, we can suppose that $R=[0,a]\times[0,b]$ and that where are
	considering non-decreasing paths instead of light-path in the definition of
	$C^{\uparrow}$. Therefore, by the union bound inequality,
	\begin{align*}
	\Pro\left(C^{\uparrow}_{\omega,\ubar{y}}(R) \right)
	&= \Pro\left( \exists (r_i,s_i)_{1 \leq i \leq k} \in \prod_{i=1}^k \omega_{y_i}, \ 0 \leq r_1
	\leq \cdots \leq r_k \leq a,\; 0 \leq s_1 \leq \cdots \leq s_k \leq b
	\right) \\
	& \leq \int_{0 \leq r_1 \leq \cdots \leq r_k \leq a} \int_{0 \leq s_1 \leq
		\cdots \leq s_k \leq b} \Pro \left( \bigcap_{i=1}^k \# \left\{\omega_{y_i}
		\cap
	[r_i,r_i+\mathrm{d}r_i] \times [s_i, s_i + \d{s_i}] \right\}=1
	\right) \\
	&\leq \int_{0 \leq r_1 \leq \cdots \leq r_k \leq a} \int_{0 \leq s_1 \leq
	\cdots \leq s_k \leq b} 2^k \d{r_1} \cdots \d{r_k} \d{s_1} \cdots \d{s_k}\\
	& \hspace{1cm} \text{(since the $\omega_{y_i}$ are independent PPPs of
	intensity $2$ on $\R \times \Z \times \R_+$)}\\
	& = \frac{(2ab)^k}{(k !)^2} \leq  \left(\frac{2e^2\,
	\Leb(R)}{k^2}\right)^k.
	\end{align*}
	In the last inequality, we used that $k ! \geq (k/e)^k$ valid for all $k\in
	\N$
	(this classical inequality can be obtained from $e^x \geq x^k/(k!)$
	evaluated
	at $x=k$).
\end{proof}
Now we give a Corollary that can be useful when dealing with domains different
from light-rectangles (the upper bound obtain is not optimal, yet enough for our purposes).
\begin{coro}\label{coro:lis}
	For any domain $D \subseteq \R^2$, any $k \in \N$ and any $\ubar{y}=(y_0,\cdots,y_k) \in \Z^{k+1}$,
	\[
	\Pro\left(C^{\uparrow}_{\omega,\ubar{y}}(D) \right) \leq  2 \, \Leb(D) \,
	\left(\frac{4e^2\, \mathrm{v}(D)^2 }{k^2}\right)^k,
	\]
	where $\mathrm{v}(D)$ is the vertical diameter of $D$ i.e the longest
	distance between two points in $D$ aligned vertically.
\end{coro}
\begin{proof}
In order to realise the event $C^{\uparrow}_{\omega,\ubar{y}}(D)$, once we have
chosen $(x_0,t_0) \in \omega_{y_0} \cap D$, then the rest of the points
$(x_1,t_1)\cdots(x_k,t_k)$ must be in the intersection between $D$ and the cone
$\{ (x,t), \ |x-x_0| \leq t-t_0\}$ which is included in a certain light-square
denoted $R_{x_0,t_0}$ whose diagonal is of length less than $2\mathrm{v}(D)$,
hence is of area less than $2 \mathrm{v}(D)^2$.
By the union bound inequality and Lemma \ref{lem:lisr},
\begin{align*}
\Pro\left(C^{\uparrow}_{\omega,\ubar{y}}(D) \right) & \leq \int_{D}
\Pro\left(C^{\uparrow}_{\omega,(y_1\cdots y_k)}(R_{x_0,t_0}) \right)
2\mathrm{d}x_0\mathrm{d}t_0\\
& \leq \int_{D}
	\left(\frac{2e^2\,  2 \mathrm{v}(D)^2}{k^2}\right)^k
	2\mathrm{d}x_0\mathrm{d}t_0 = 2 \, \Leb(D) \,
		\left(\frac{4e^2\, \mathrm{v}(D)^2 }{k^2}\right)^k.
\end{align*}
\end{proof}

\section{Compactness for asymptotically continuous functions}\label{sec:compact}
In this section, we show a generalisation of Arzelà-Ascoli theorem,
that gives sufficient conditions for ``almost continuous functions''
(e.g.  sequences of functions with jumps of size tending to $0$) to
converge uniformly on all compact sets.
\begin{prop}\label{prop:ascoli}
	Let $(f_n)_{n \in \N}$ be a sequence of functions from a separable metric
	space $(E,d)$ to a complete metric space $(F,d')$ such that:
	\begin{enumerate}
		\item \emph{Asymptotic equi-continuity}: For all $x\in E$ and all
		$\eps>0$, there exists $\delta>0$ such that
		\begin{equation}\label{eq:hypasymequi}
		\limsup_{n \to \infty}
		\sup_{\substack{y \in E \\ d(x,y) \leq \delta}}d'(f_n(x),f_n(y)) \leq
		\eps.
		\end{equation}
		\item \emph{Pointwise relative compactness}: For all $x \in E$,
		the sequence $(f_n(x))_{n \in \N}$ is contained in a compact set of $F$.
	\end{enumerate}
	Then, for any subsequence $(n_k)_{k \in \N}$, $(f_{n_k})_{k \in \N}$ has a
	subsequence that converges uniformly on all compact subsets of $E$ to a
	function $f:(E,d)\mapsto (F,d')$. Moreover, any limit point is continuous.
\end{prop}

\begin{proof}
	For the sake of simplicity and since any subsequence $(f_{n_k})_{k \in \N}$
	still satisfies assumptions 1
	and
	2, we can assume that $(f_{n_k})_{k \in \N} =
	(f_{n})_{n \in \N}$ without loss of generality.

	Let $E_0$ be a dense countable subset of $E$. By pointwise relative
	compactness and
	a diagonal extraction argument, we can find a subsequence
	$(n_l)_{l \in \N}$ such that for every $x\in E_0$,
	$(f_{n_l}(x))_{l \in \N}$ converges in $F$.  Let us show that
	actually, for all $x\in E$, $(f_{n_l}(x))_{l \in \N}$ is a
	Cauchy sequence, hence converges in $F$. Let $x \in E$ and
	$\eps>0$. By assumption, there exists $\delta>0$ such that
	\eqref{eq:hypasymequi} is satisfied. By density, we can find
	$x_0 \in E_0$ such that $d(x,x_0)\leq \delta$. As
	$(f_{n_l}(x_0))_{l \in \N}$ converges, it is a Cauchy sequence
	so for all $l,m$ large enough,
	$d'(f_{n_l}(x_0),f_{n_{m}}(x_0)) \leq \eps$ and thus
	\begin{align*}
	d'(f_{n_l}(x),f_{n_{m}}(x))	&\leq d'(f_{n_l}(x),f_{n_{l}}(x_0)) +
	d'(f_{n_l}(x_0),f_{n_{m}}(x_0)) +
	d'(f_{n_m}(x_0),f_{n_{m}}(x))\\
	& \leq 3
	\eps ,
	\end{align*}
	for $l,m$ large enough by \eqref{eq:hypasymequi}. Let us call $f$ the
	pointwise limit. By taking the limit in
	\eqref{eq:hypasymequi}, we get
	immediately that any such limit point is continuous.

	Now, let $K$ be a compact subset of $E$ and let us show that $f_{n_l}$
	converges to $f$ uniformly on $K$. Let $\eps>0$. By compactness and
	asymptotic equi-continuity assumption, we can find a
	covering of $K$ by a finite number $p \in \N$ of balls of centers
	$x_1,\cdots x_p$ and radius
	$\delta_1,\cdots,\delta_p$ such that \eqref{eq:hypasymequi} is satisfied
	with $(x,\delta)=(x_i,\delta_i)$ for any $i \in  \{1,\cdots p\}$. Therefore, we
	can find $N \in \N$ such that for all $l \geq N$,
	\begin{equation}\label{eq:distance_balls}
		 \forall i \in
	\{1,\cdots,p\} \quad \forall y \in \mathcal{B}(x_i,\delta)
	\qquad
	d'(f_{n_l}(x_i),f_{n_l}(y)) \leq \eps .
	\end{equation}
	 Moreover,
	by point-wise convergence we can assume that for all $l \geq N$ and all $i
	\in \{1, \cdots, p \}$,
	$d'(f_{n_l}(x_i),f(x_i)) \leq \eps$. Therefore, for all
	$l \geq N$ and all $y \in K$, if we choose the index $i$ such that $d(y,x_i) \leq \delta_i$, then
	\begin{align*}
	d'(f_{n_l}(y),f(y)) & \leq d'(f_{n_l}(y),f_{n_l}(x_i)) +
	d'(f_{n_l}(x_i),f(x_i)) +d'(f(x_i),f(y))\\
	&  \leq 3 \, \eps,
	\end{align*}
	where we used \eqref{eq:distance_balls} and point-wise convergence in the last inequality.
\end{proof}

\section*{Acknowledgement}
I would like to thank my advisor Fabio Toninelli for introducing me to the
domain of hydrodynamic limits, for frequent discussions and for very careful
proofreading. I also want to thank Xufan Zhang for explaining details about his
article~\cite{zhang2018domino}.
This work was partially supported by  ANR-15-CE40-0020-03 Grant LSD.


\bibliographystyle{plain}
\bibliography{biblio}

\Addresses

\end{document}